\documentclass[12pt]{amsart}
\usepackage{amssymb}
\usepackage{amsmath}
\usepackage{amsthm}
\usepackage{graphicx}
\usepackage{color}
\usepackage{cancel}
\usepackage[usenames,dvipsnames]{xcolor}
\usepackage[normalem]{ulem}
\usepackage{url}
\usepackage{paralist}
\usepackage{hyperref}
\usepackage{cite}

\numberwithin{equation}{section}

\newcommand{\R}{\mathbb{R}}
\newcommand{\T}{\mathbb{T}}
\newcommand{\N}{\mathbb{N}}
\newcommand{\abs}[1]{\left| #1\right|}

\newcommand{\set}[1]{\left\{#1\right\}}

\newtheorem{theorem}{Theorem}[section]

\newtheorem{proposition}[theorem]{Proposition}

\theoremstyle{definition}

\title[Focusing Singularity in a Derivative NLS Equation]
{Focusing Singularity in a Derivative Nonlinear Schr\"odinger Equation}

\author[Liu]{Xiao Liu}
\email{liuxiao@math.toronto.edu}
\address{Department of Mathematics, University of Toronto, 40
  St. George St., Room 6290, Toronto, Ontario M5S 2E4, Canada}

\author[Simpson]{Gideon Simpson}
\email{gsimpson@umn.edu}
\address{School of Mathematics, University of Minnesota, 206 Church St.
  SE, 127 Vincent Hall, Minneapolis, MN 55455, USA}
 \thanks{G.S. was supported by NSERC.  His contribution to this work
   was completed under the NSF PIRE grant OISE-0967140 and the DOE grant DE-SC0002085.}

\author[Sulem]{Catherine Sulem}
\email{sulem@math.toronto.edu}
\address{Department of Mathematics, University of Toronto, 40
  St. George St., Room 6290, Toronto, Ontario M5S 2E4, Canada}
 \thanks{C.S. is  partially  supported by NSERC through grant number 46179-11.}


\date{\today}

\begin{document}

\maketitle

\begin{abstract}
  We present a numerical study of a derivative nonlinear Schr\"odinger
  equation with a general power nonlinearity,
  $\abs{\psi}^{2\sigma}\psi_x$.  In the $L^2$-supercritical regime,
  $\sigma>1$, our simulations indicate that there is a finite time
  singularity. We obtain a precise description of the local structure
  of the solution in terms of blowup rate and asymptotic profile, in a
  form similar to that of the nonlinear Schr\"odinger equation with
  supercritical power law nonlinearity.

\end{abstract}

\section{Introduction}
We consider the derivative nonlinear Schr\"{o}dinger (DNLS) equation
\begin{equation}
  \label{eqn:DNLS}
  i \partial_t \Psi+ \Psi_{xx}+i(|\Psi|^{2} \Psi)_x=0,\quad x\in \mathbb{R}, \quad t\in \mathbb{R}
\end{equation}
with initial condition $\Psi(x,0) = \Psi_0(x)$.  Under a long
wavelength approximation, this nonlinear dispersive wave equation is a
model for Alfv{\'e}n waves in plasma physics,
\cite{MJOLHUS1976,Passot1993,Sulem1999}.  Using a Gauge
transformation,
\begin{equation}
  \psi=\Psi(x)\exp\left\{\frac{1}{2}\int_{-\infty}^x|\Psi(\eta)|^2d\eta\right\},
\end{equation}
\eqref{eqn:DNLS} takes the form
\begin{equation}
  \label{eqn:DNLS2}
  i \partial_t \psi+ \psi_{xx}+i|\psi|^{2}\psi_x=0,\quad x\in\mathbb{R}, \quad  t\in \mathbb{R}.
\end{equation}
This equation appeared in studies of ultrashort optical pulses,
\cite{Agrawal2006,Moses2007}. The latter equation admits a Hamiltonian
form
\begin{equation}
  \label{eqn:Ham}
  \partial_t \psi = - i\frac{\delta E}{\delta \bar\psi}
\end{equation}
with the Hamiltonian
\begin{equation}
  \label{eqn:Energy_dnls}
  E=\frac{1}{2}\int_{-\infty}^\infty|\psi_x|^2
  dx+\frac{1}{4}\Im\int_{-\infty}^\infty |\psi|^{2}\bar{\psi}\psi_x
  dx.
\end{equation}

The well-posedness of \eqref{eqn:DNLS} has been studied by many
authors.  One of the earliest results is due to Tsutsumi and Fukuda,
who proved local well-posedness on both $\R$ and $\T$ in the Sobolev
space $H^s$, provided $s > 3/2$ and the data is sufficiently small,
\cite{Tsutsumi1980}.  This was subsequently refined by Hayashi and
Ozawa, who found that for initial conditions satisfying
\begin{equation}
  \label{e:smallness}
  \|u_0\|_{L^2}<\sqrt{2\pi},
\end{equation}
the solution was global in $H^s$ for $s\in \N$, \cite{Hayashi1992}.
More recently, the global in time result for data satisfying
\eqref{e:smallness} was extended to all $H^s$ spaces with $s > 1/2$,
\cite{Colliander2002}.

Local well posedness has also been studied with additive terms,
\cite{Difranco2008}, and with more general nonlinearities,
\cite{Kenig:1998wr,Linares2009}.  However, an outstanding problem for
DNLS is to determine the fate of large data, violating
\eqref{e:smallness}.  At present, there is neither a result on global
well posedness, nor is there a known finite time singularity.

In this paper, we consider a generalized derivative nonlinear
Schr\"{o}dinger (gDNLS) equation of the form ($\sigma\geq 1$)
\begin{equation}
  \label{eqn:gDNLS}
  i \partial_t\psi+\psi_{xx}+i|\psi|^{2\sigma}\psi_x=0, \quad x\in \mathbb{R}, \quad t\in \mathbb{R}.
\end{equation}
It also has a Hamiltonian structure with generalized energy:
\begin{equation}
  \label{eqn:Energy}
  E=\frac{1}{2}\int_{-\infty}^\infty|\psi_x|^2
  dx+\frac{1}{2(\sigma+1)}\Im\int_{-\infty}^\infty
  |\psi|^{2\sigma}\bar{\psi}\psi_x dx, 
\end{equation}
along with the mass and momentum invariants:
\begin{align}
  \label{eqn:Mass_dnls}
  M&=\frac{1}{2}\int_{-\infty}^\infty|\psi|^2 dx,
  \\
  \label{eqn:Momentum_dnls}
  P&=-\frac{1}{2}\Im\int_{-\infty}^\infty\bar{\psi}\psi_x dx.
\end{align}
For $\sigma> 5/2$, the local well-posedness for \eqref{eqn:gDNLS} is
proved for initial data in $H^{1/2}$ intersected with an appropriate
Strichartz space, \cite{Hao2007}.

For all values of $\sigma$, we have the scaling property that if
$\psi(x,t)$ is a solution of \eqref{eqn:gDNLS}, then so is
\begin{equation}
  \label{eqn:scaling_prop}
  \psi_{\lambda}(x,t)=\lambda^{-\frac{1}{2\sigma}} \psi\left(\lambda^{-1}{x},\lambda^{-2}{t}\right).
\end{equation}
Hence, \eqref{eqn:gDNLS} is $L^2$-critical for $\sigma = 1$ and
$L^2$-supercritical for $\sigma>1$.  It is well known that the
$L^2$-supercritical NLS equation with power law nonlinearity has
finite time singularities for sufficiently large data.  The goal of
this work is to explore the potential for collapse in
\eqref{eqn:gDNLS} when $\sigma >1$.

Our simulations and asymptotics indicate that there is collapse, and
we give an accurate description of the nature of the focusing
singularity with generic ``large'' data.  This is done using the
dynamic rescaling method, which was introduced to study the local
structure of NLS singularities,\cite{McL1986,Sulem1999}. The idea
behind dynamic rescaling is to introduce an adaptive grid through a
nonlinear change of variables based on the scaling invariance of the
equation.  Up to a rescaling, we observe numerically for $\sigma>1$,
the solution blows up locally like
\begin{equation}
  \label{e:universal_blowup}
  \small
  \psi(x,t) \sim \left(\frac{1}{2\alpha(t^{*}-t)}\right)^{\frac{1}{4\sigma}}
  Q\left(\frac{x-x^*}{\sqrt{2\alpha(t^{*}-t)}}+\frac{\beta}{\alpha}\right)e^{i(\theta+
    \frac{1}{2\alpha}\ln\frac{t^{*}}{t^{*}-t})},
\end{equation}
where $t^*$ is the singularity time and $x^*$ is the position of
$\max_x{|\psi(x,t^*)|}$.  The function $Q(\xi)$ is a complex-valued
solution of the equation
\begin{equation}
  \label{e:blowup_soln}
  Q_{\xi\xi} - Q +i\alpha(\tfrac{1}{2\sigma}Q +\xi Q_\xi)-i\beta Q_{\xi}+ i|Q|^{2\sigma}Q_\xi = 0,
\end{equation}
with amplitude that decays monotonically as $\xi \to \pm \infty$. The
coefficients $\alpha$ and $\beta$ are real numbers, and our
simulations show that they depend on $\sigma$ but not on the initial
condition.  We also observe that $\alpha$ decreases monotonically as
$\sigma\to 1$, while $\beta$ first decreases then increases (Figure
\ref{fig:bvp_a_b}).  This observed blowup is analogous to that of NLS
with supercritical power nonlinearity in terms of the blowup speed and
asymptotic profile, \cite{McL1986}.

Our paper is organized as follows.  In section \ref{s:sigma=2}, we
consider gDNLS with quintic nonlinearity ($\sigma=2$). We first
present a direct numerical simulation that suggests there is, indeed,
a finite time singularity.  We then refine our study by introducing
the dynamic rescaling method.  Section \ref{s:generalsigma} shows our
results for $\sigma\in(1,2)$. In the Section \ref{s:profile}, we
discuss the asymptotic profile of gDNLS.  We discuss our results in
Section \ref{s:discussion}, and their implications for blowup in DNLS.
Details of our numerical methods are described in the Appendix.

\section{gDNLS with Quintic Nonlinearity}
\label{s:sigma=2}

We first performed a direct numerical integration of the the gDNLS
equation with quintic nonlinearity ($\sigma = 2$) and initial
condition $\psi_0 = 3 e ^{-2x^2}$. We integrated the equation using a
pseudo-spectral method with an exponential time-differencing
fourth-order Runge-Kutta algorithm (ETDRK4), augmenting it with the
numerical stabilization scheme presented in \cite{Kassam2005} to
better resolve small wave numbers. The computation was performed on
the interval $[-8,8)$, with $2^{12}$ grid points and a time step on
the order of $10^{-6}$.  Figure \ref{fig:u_without_dyn_res} shows that
the wave first moves rightward and then begins to separate. Gradually,
the leading edge of one wave sharpens and grows in time.  The norms
$\|\psi_x(\cdot ,t)\|_{L^2}$ and $\|\psi_{xx}(\cdot ,t)\|_{L^2}$ grow
substantially during the lifetime of the simulation.  The norm
$\|\psi_x(\cdot ,t)\|_{L^2}$ increases from 4 to about 18 while
$\|\psi_{xx}(\cdot,t)\|_{L^2}$ increases from about 10 to 1720.  This
growth in norms is a first indication of collapse.  In contrast, the
$L^\infty$ norm has only increased from 3 to 4.07. As a measure of the
precision of the simulation, the energy remains equal to $ 7.976$,
with three digits of precision, up to time $t=0.009$.

\begin{figure}
  \includegraphics[width=\linewidth]{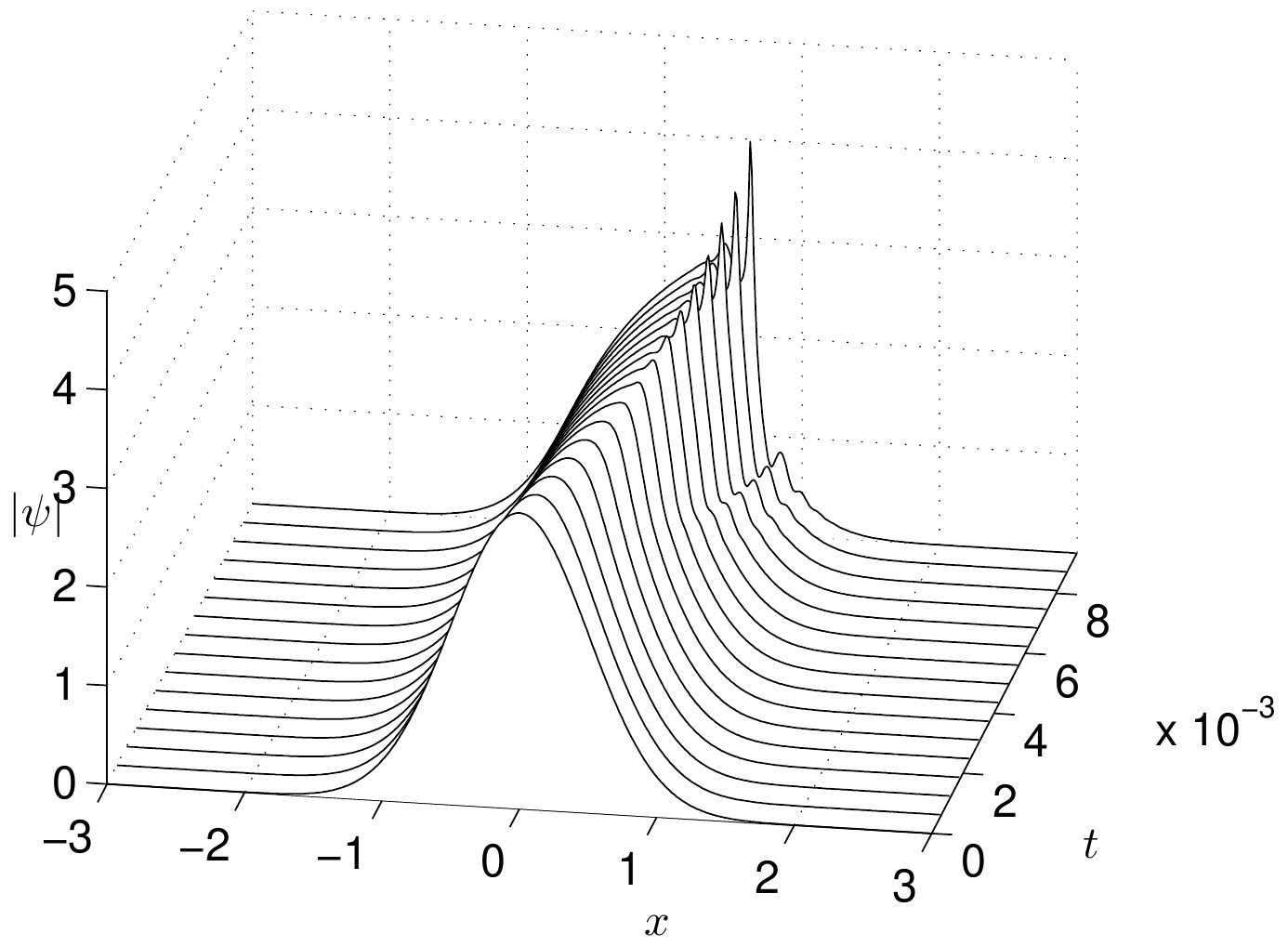}

  \caption{$|\psi(x,t)|$ versus $x$ and $t$.}
  \label{fig:u_without_dyn_res}
\end{figure}

\subsection{Dynamical Rescaling Formulation}

To gain additional detail of the blowup, we employ the dynamic
rescaling method.  Based on the scaling invariance of the equation
\eqref{eqn:scaling_prop}, we define the new variables:
\begin{equation}
  \label{eqn:u_psi}
  \psi(x,t) = L^{-\frac{1}{2\sigma}}(t) u(\xi,\tau),\quad \xi=\frac{x - x_0(t)}{L(t)}\quad \mbox{and} \quad \tau =\int_0^t\frac{ds}{L^2(s)},
\end{equation}
where $L(t)$ is a length scale parameter chosen such that a certain
norm of the solution remains bounded for all $\tau$.  The parameter
$x_0$ will be chosen to follow the transport of the solution.
Substituting the change of variables into \eqref{eqn:gDNLS} gives
\begin{equation}
  \label{eqn:dyn_res_DNLS}
  \begin{split}
    &iu_\tau + u_{\xi\xi} + i a(\tau)(\tfrac{1}{2\sigma}u+\xi u_\xi) -i b(\tau) u_\xi + i|u|^{2\sigma} u_\xi = 0,\\
    &u(\xi,0)=\psi(x,0),
  \end{split}
\end{equation}
where
\begin{subequations}
  \label{eqn:ab}
  \begin{equation}
    \label{eqn:a_b}
    a(\tau) = -L(t)\dot{L}(t) = -\frac{d \ln L}{d \tau},\\
  \end{equation}
  \begin{equation}
    \label{eqn:a_b2}
    b(\tau)  = L(t) \frac{d x_0}{d t}.
  \end{equation}
\end{subequations}
There are several possibilities for the choice of $L(t)$.  We define
$L(t)$ such that $\| u_\xi(\cdot,\tau)\|_{L^2}$ remains constant.
From \eqref{eqn:u_psi},
\[
\| \psi_{x} \|_{L^{2p}} =L^{-\frac{1}{q}} \| u_{\xi} \|_{L^{2p}},
\quad q =
\left(1+\frac{1}{2}\left(\frac{1}{\sigma}-\frac{1}{p}\right)\right)^{-1},
\]
and
\begin{equation*}
  L(t)=\|u_{\xi}(\cdot,0)\|_{L^{2p}}^q \|\psi_{x}(\cdot,t) \|_{L^{2p}}^{-q},
\end{equation*}
leading to an integral expression for the function $a(\tau)$ defined
in \eqref{eqn:a_b}:
\begin{equation}
  \label{eqn:a}
  a(\tau) = -q \| u_{\xi}(\cdot,0) \|_{L^{2p}}^{-2p} \int \Re\left\{(\bar{u}_{\xi}^p u_{\xi}^{p-1})_{\xi}(i u_{\xi\xi} - |u|^{2\sigma}u_\xi) \right\} d\xi,
\end{equation}
Ideally, we would like to choose $x_0(t)$ to follow the highest
maximum of the amplitude of the solution.  After several attempts, we
found that choosing
\begin{equation}
  x_0=\frac{\int x |\psi_x|^{2p} dx}{\int |\psi_x|^{2p} dx}
\end{equation}
follows the maximum of the amplitude in a satisfactory way.  The
coefficient $b$ takes the form
\begin{equation} \label{def-b} b(\tau) = 2p\frac{\int \Re\{ (\xi
    |u_{\xi\xi}|^{2p-2} \bar{u}_{\xi\xi})_{\xi\xi} (i u_{\xi\xi} -
    |u|^{2\sigma} u_{\xi})\} d\xi}{\int |u_{\xi\xi}|^{2p} d\xi}.
\end{equation}
In summary, we now have a system of evolution equations for the
rescaled solution $u$:
\begin{subequations}
  \label{eqn:dyn_res_DNLS_system_u_a_b}
  \begin{align}
    u_\tau &= i u_{\xi\xi} - a(\tau)(\tfrac{1}{2\sigma}u+\xi u_\xi) +
    b(\tau)u_\xi - |u|^{2\sigma} u_\xi \\
    a(\tau)&= -q\| u_{\xi}(0) \|_{L^{2p}}^{-2p} \int
    \Re\left\{(\bar{u}_{\xi}^p u_{\xi}^{p-1})_{\xi}(i
      u_{\xi\xi} - |u|^{2\sigma}u_\xi)\right\}d\xi,\\
    b(\tau) &= 2p\frac{\int \Re\{ (\xi |u_{\xi\xi}|^{2p-2}
      \bar{u}_{\xi\xi})_{\xi\xi} (i u_{\xi\xi} - |u|^{2\sigma}
      u_{\xi})\} d\xi}{\int |u_{\xi\xi}|^{2p} d\xi}.
  \end{align}
\end{subequations}
where $p\in\mathbb{N}$ and $q =
(1+\frac{1}{2}(\frac{1}{\sigma}-\frac{1}{p}))^{-1}$.  The scaling
factor $L(\tau)$ is computed by integration of \eqref{eqn:a_b}.

We expect $u(\xi,\tau)$ to be defined for all $\tau$ and that
$L(\tau)\to 0 $ fast enough so that $\tau \to \infty$ as $t \to
t^*$. The behavior of $a(\tau)$ and $u(\xi,\tau)$ for large $\tau$
will give us information on the scaling factor $L(\tau)$ and the
limiting profile. Under this rescaling, the invariants
\eqref{eqn:Energy}, \eqref{eqn:Mass_dnls} and
\eqref{eqn:Momentum_dnls} become
\begin{subequations}
  \label{eqn:invariants_dyn_res}
  \begin{align}
    \label{eqn:Energy_dyn_res}
    E(\psi(t))&=E(\psi(0)) = L(t)^{-\frac{1}{\sigma}-1} E(u(\tau)),\\
    \label{eqn:Mass_dyn_res}
    M(\psi(t)) &= M(\psi(0)) = L(t)^{-\frac{1}{\sigma}+1} M(u(\tau)),\\
    \label{eqn:Momentum_dyn_res}
    P(\psi(t)) &= P(\psi(0)) = L(t)^{-\frac{1}{\sigma}}P(u(\tau)).
  \end{align}
\end{subequations}
These relations allow us to compute $L(t)$ in different ways and check
the consistency of our calculations.

\subsection{Singularity Formation}
\label{ss:dyn_res_s=2}

Integrating \eqref{eqn:dyn_res_DNLS} using the ETDRK4 algorithm, we
present our results for two families of initial conditions:
\begin{equation}
  \label{gauss}
  \begin{split}
    \psi_0(x) = A_0 e^{-2x^2} \quad\text(Gaussian)
  \end{split}
\end{equation}
and
\begin{equation} \label{lor} \psi_0(x) = \frac{A_0}{1+9x^2}
  \quad\text(Lorentzian).
\end{equation}
with various values of the amplitude coefficient $A_0$.  We observed
that for both families, the corresponding solutions present similar
local structure near the blowup time.

In these simulations, there are typically $2^{18}$ points in the
domain $[-l, l)$ with $l=1024$.  While the exponential time
differencing removes the stiffness associated with the second
derivative term, the advective coefficient,
\[
(a \xi - b + \abs{u}^{2\sigma})u_{\xi},
\]
constrains our time step through a CFL condition.  Near the origin,
where $u$, initially, can have an amplitude as large as $A_0=4$, we
have a coefficient for the quintic case that can be $\sim 100$.  Far
from the origin, $a\xi$ is the dominant term, and this coefficient can
be $\sim 1000$.  Consequently, the time step must be at least two to
three orders of magnitude smaller than the grid spacing.  For the
indicated spatial resolution, $\Delta \xi \sim 10^{-2}$, we found it
was necessary to take $\Delta \tau\sim 10^{-7}$ to ensure numerical
stability.

Figure \ref{fig:A=3_water_fall} shows the evolution of $|u(\xi,\tau)|$
with the Gaussian initial condition, $A_0=3$ and $L(0)= 0.2$.  The
rescaled wave $u(\xi,\tau)$ separates into three pieces.  The middle
one stays in the center of the domain while the other two waves move
away from the origin on each side.  Figure
\ref{fig:A=3_H1_H2_amplitude} shows that $\max_\xi |u(\xi,\tau)|$
slightly decreases with $\tau$. Due to our choice of the rescaling
factor, $\|u_\xi(\cdot,\tau)\|_{L^2}$ is constant. We also observe
that $\|u_{\xi\xi}(\cdot,\tau)\|_{L^2}$ remains bounded.  Returning to
the primitive variables, the maximum integration time $\tau=6$
corresponds to $t=0.0098$, $\|\psi_x\|_{L^2}=334.58$ and
$\|\psi_{xx}\|_{L^2}=1.56\times10^{6}$.  For the Lorentzian initial
condition \eqref{lor} with $A_0=3$ and $L(0) = 1$, the maximum
integration time $\tau=0.16$ corresponds to $t=0.008$,
$\|\psi_x\|_{L^2}=109.8$ and $\|\psi_{xx}\|_{L^2}=1.17\times10^5$.

\begin{figure}[h]
  \begin{center}
    \includegraphics[width=0.75\linewidth]{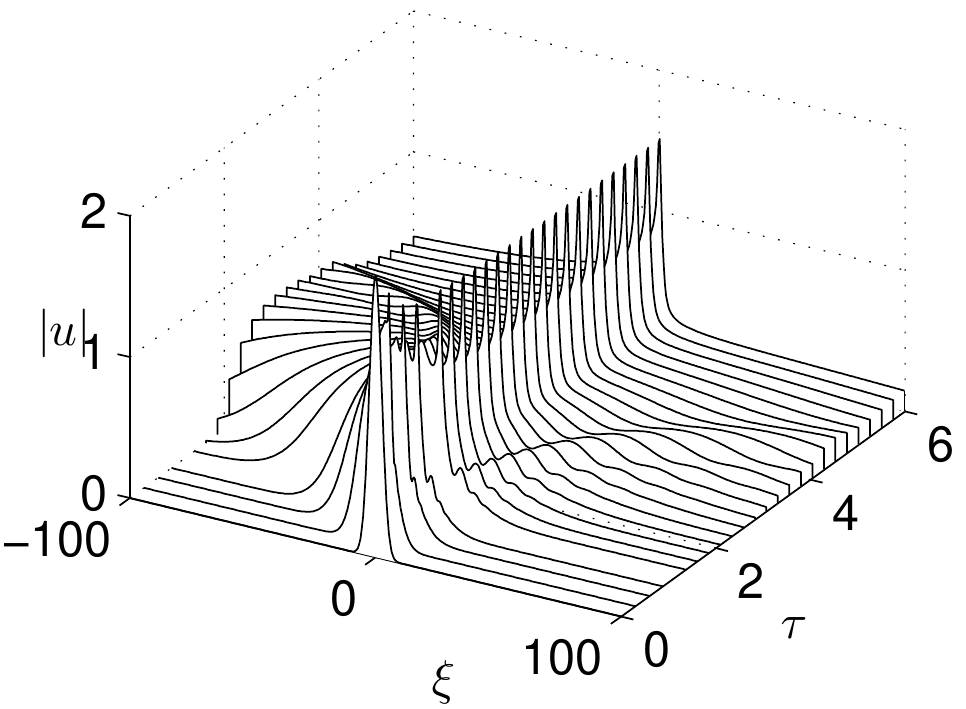}
    \caption{ $|u|$ versus $\xi$ and $\tau$, for initial condition
      \eqref{gauss} with $A_0 =3$, and nonlinearity $\sigma=2$.}
    \label{fig:A=3_water_fall}
  \end{center}
\end{figure}

\begin{figure}[h]
  \begin{center}
    \includegraphics[width=0.75\linewidth]{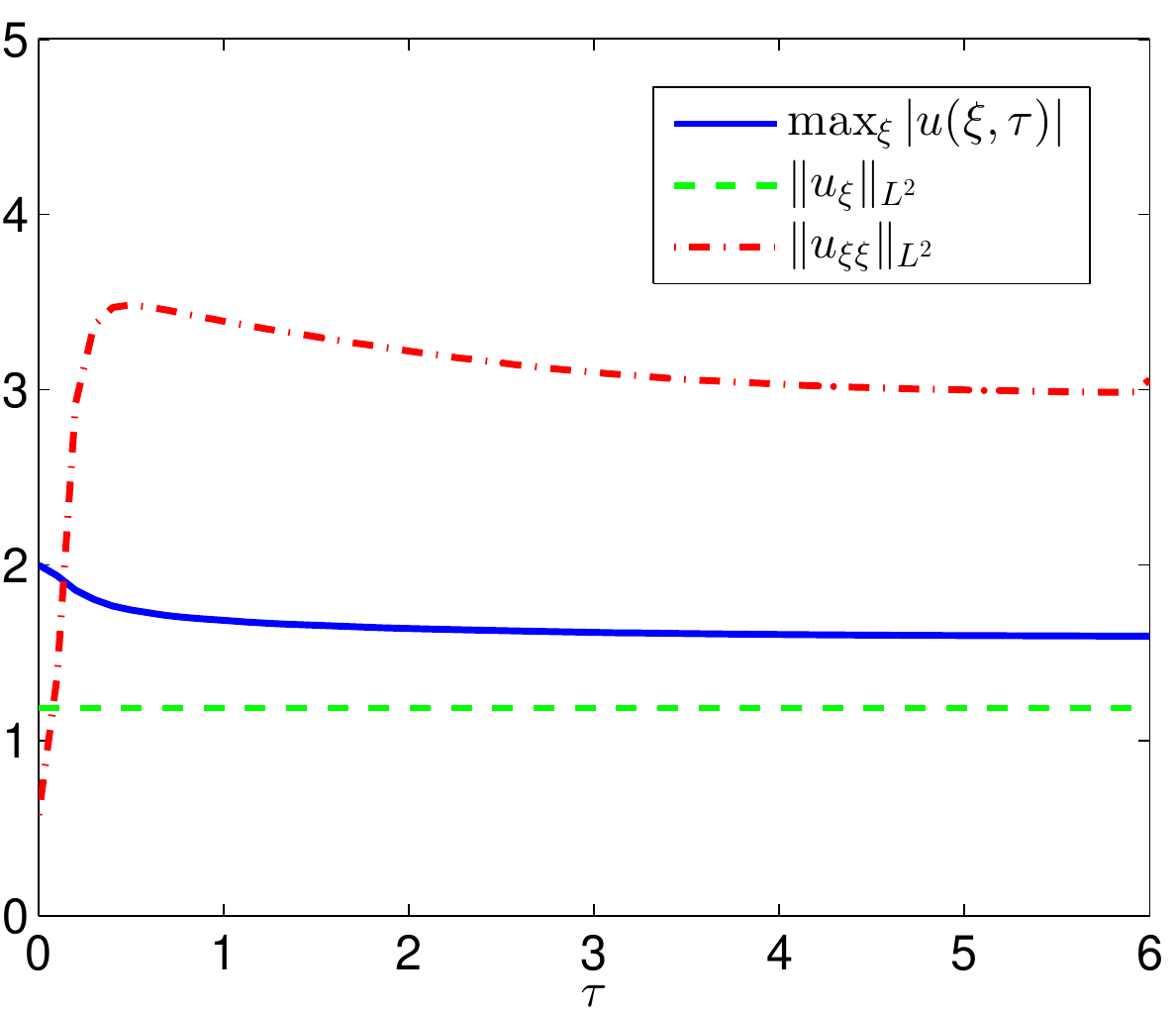}
    \caption{ $\max_\xi{|u(\xi,\tau)|}$,
      $\|u_\xi(\cdot,\tau)\|_{L^2}$, and
      $\|u_{\xi\xi}(\cdot,\tau)\|_{L^2}$ versus $\tau$, same initial
      conditions as in Fig.  \eqref{fig:A=3_water_fall}.  }
    \label{fig:A=3_H1_H2_amplitude}
  \end{center}
\end{figure}

For a detailed description of singularity formation, we turn to the
evolution of the parameters $a$ and $b$ as functions of $\tau$.  For
both families of initial conditions \eqref{gauss} and \eqref{lor} and
amplitudes $A_0=2,3,4$, we observe that $a$ and $b$ tend to constants
$A$ and $B$ as $\tau$ gets large.  Figures \ref{fig:A=3_a} and
\ref{fig:A=3_b} respectively display $a$ and $b$ versus $\tau$ for
initial conditions \eqref{gauss} and \eqref{lor} and $A_0=3$.

\begin{figure}
  \begin{center}
    \includegraphics[width=0.45\linewidth]{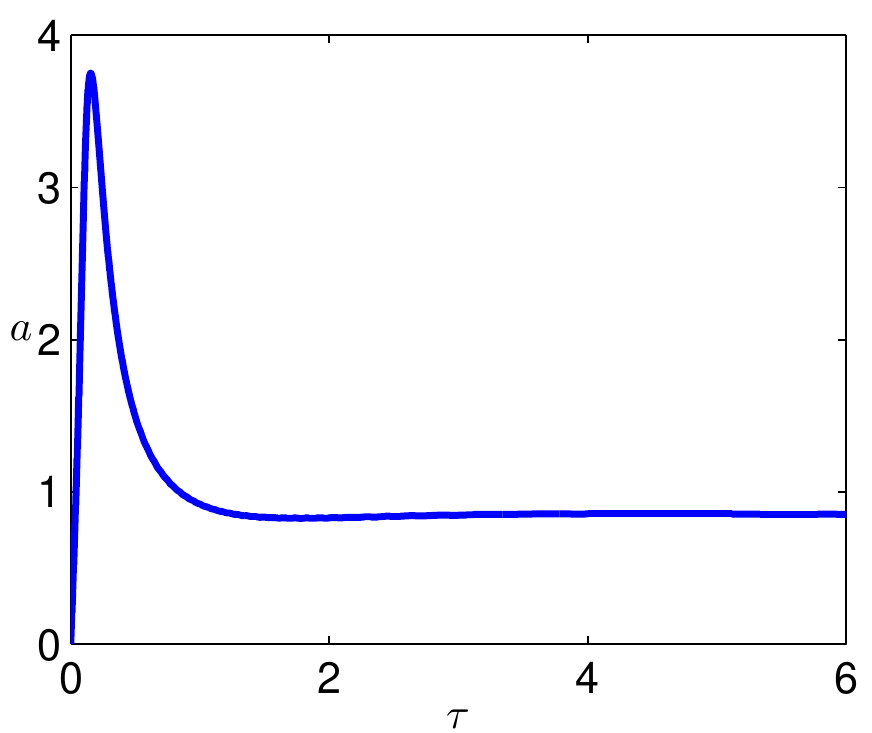}
    \includegraphics[width=0.45\linewidth]{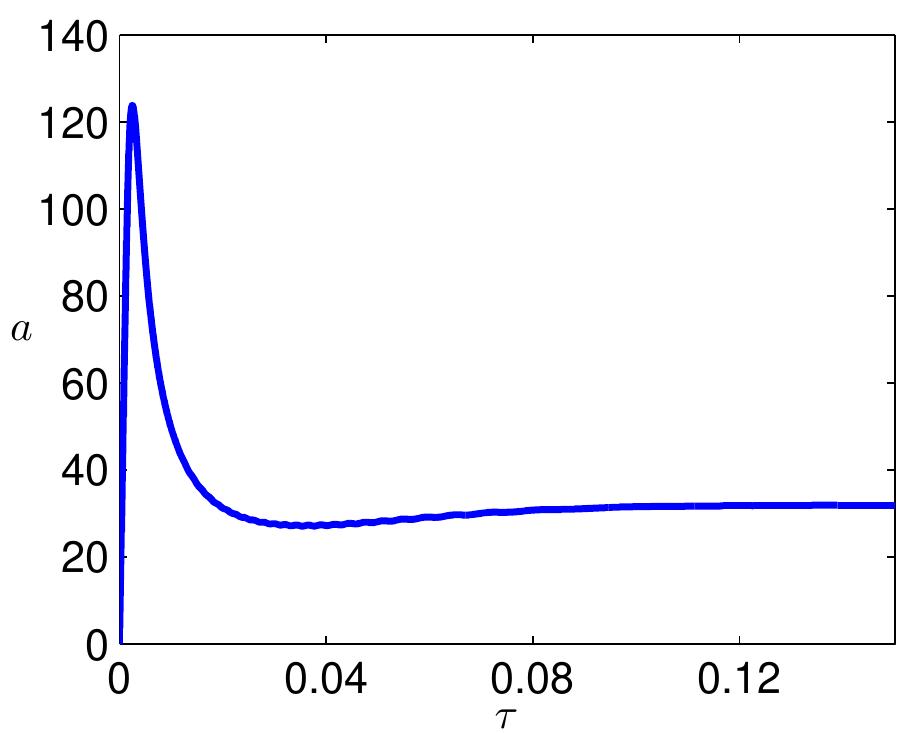}
    \caption{ Time evolution of $a(\tau)$ , with initial condition
      \eqref{gauss} (left) and \eqref{lor} (right), $A_0 =3$.}
    \label{fig:A=3_a}
  \end{center}
\end{figure}

\begin{figure}
  \begin{center}
    \includegraphics[width=0.45\linewidth]{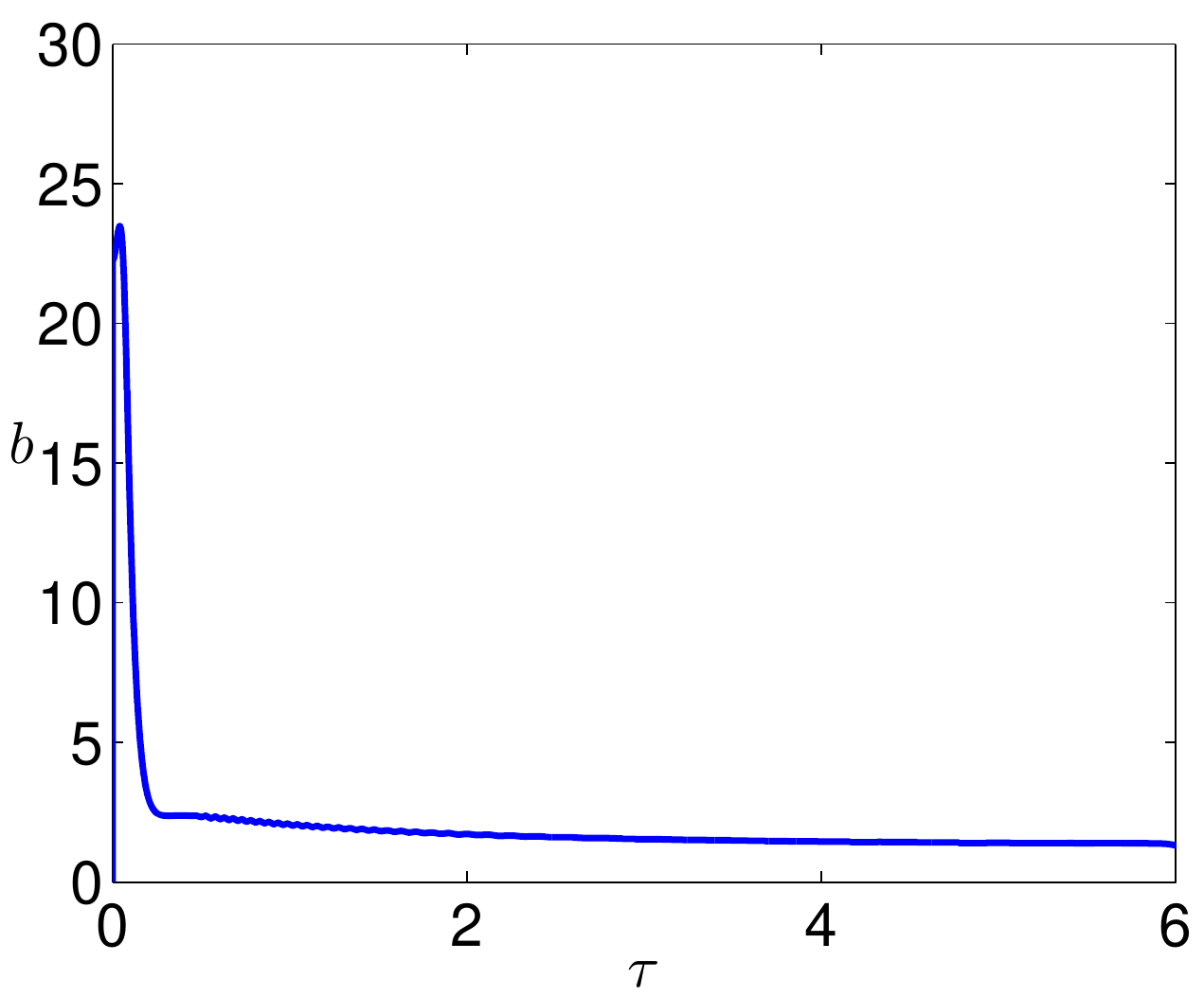}
    \includegraphics[width=0.47\linewidth]{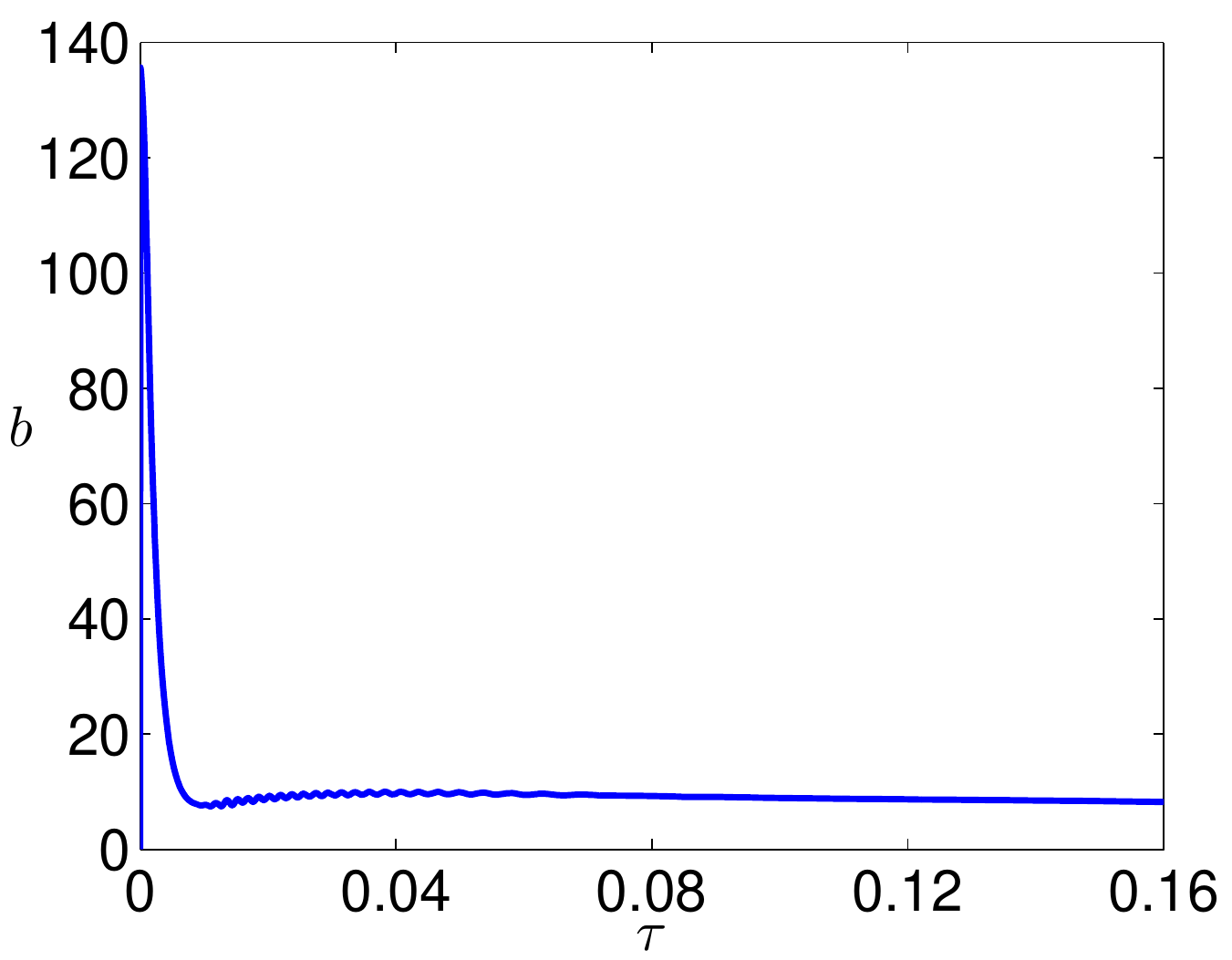}
    \caption{ Time evolution of $b(\tau)$ , with initial condition
      \eqref{gauss} (left) and \eqref{lor} (right), $A_0 =3$.}
    \label{fig:A=3_b}
  \end{center}
\end{figure}

Turning to the limiting profile of the solution, we write $u$ in terms
of amplitude and phase in the form $u \equiv |u|e^{i\phi(\xi,\tau)}$,
Figure \ref{fig:A=3_abs_u} shows that $|u|$ tends to a fixed profile
as $\tau$ increases. Moreover, as shown in Figure
\ref{fig:A=3_phase_xi=0}, the phase at the origin is linear for $\tau$
large enough, namely $\phi(0,\tau) \sim C \tau$. The constant $C$ is
obtained by fitting the phase at the origin using linear least squares
over the interval $[\tau_{max}/2,\tau_{max}]$ After extracting this
linear phase, the rescaled solution $u$ tends to a time-independent
function, (Figure \ref{fig:A=3_phase}):
\begin{equation}
  \label{eqn:u_1}
  u(\xi,\tau) \sim S(\xi)e^{i(\theta+ C\tau) }\quad\text{as}\quad
  \tau\rightarrow\infty.
\end{equation}

\begin{figure}
  \begin{center}
    \includegraphics[width=0.45\linewidth]{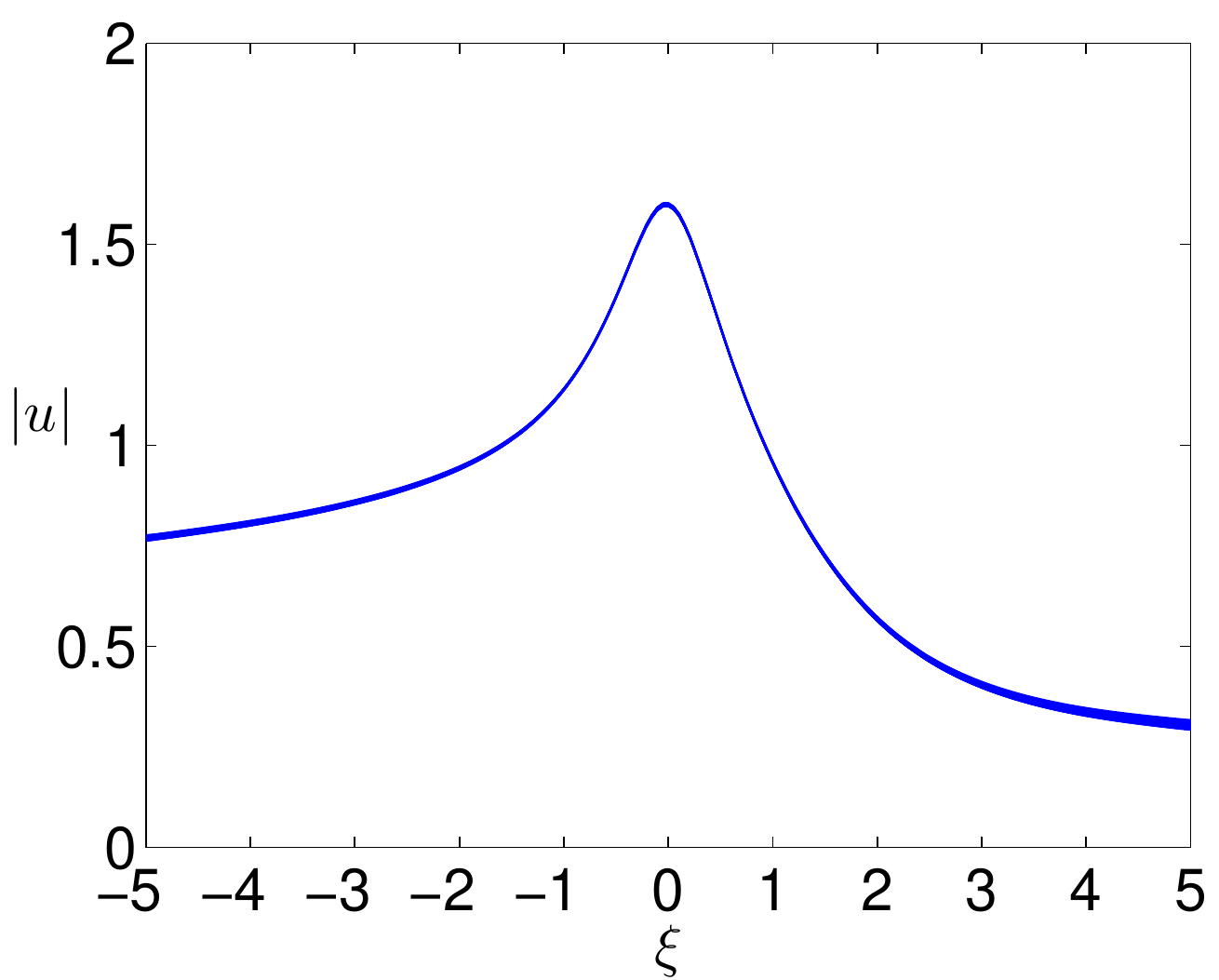}
    \includegraphics[width=0.45\linewidth]{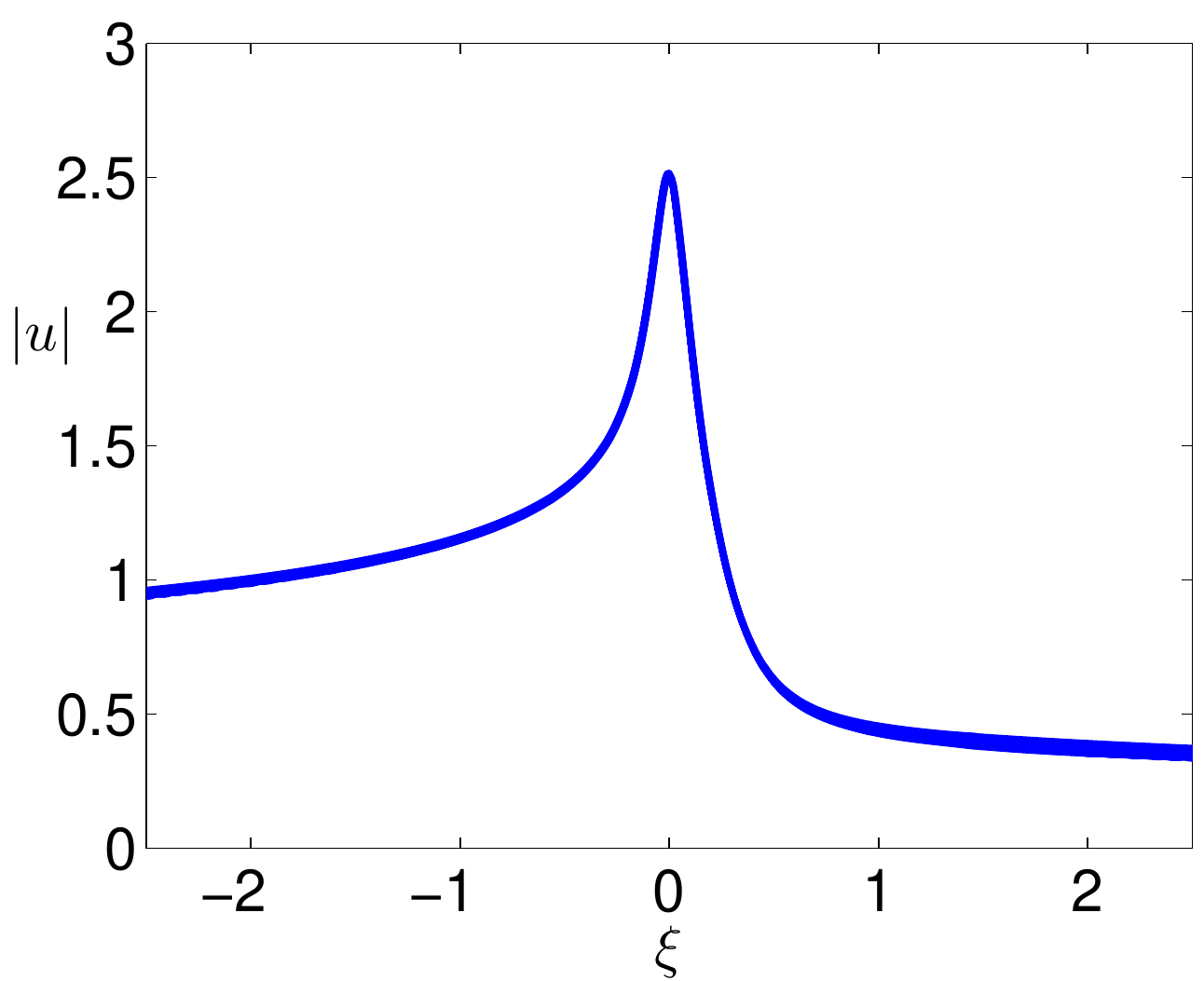}
    \caption{$|u(\xi,\tau)|$ versus $\xi$ , with initial condition
      \eqref{gauss} for $\tau$ from 4.05 to 6 with an increment of
      0.08. (left) and initial condition \eqref{lor} for $\tau$ from
      0.12 to 0.16 with an increment of 0.002. (right).}
    \label{fig:A=3_abs_u}
  \end{center}
\end{figure}

\begin{figure}
  \begin{center}
    \includegraphics[width=0.45\linewidth]{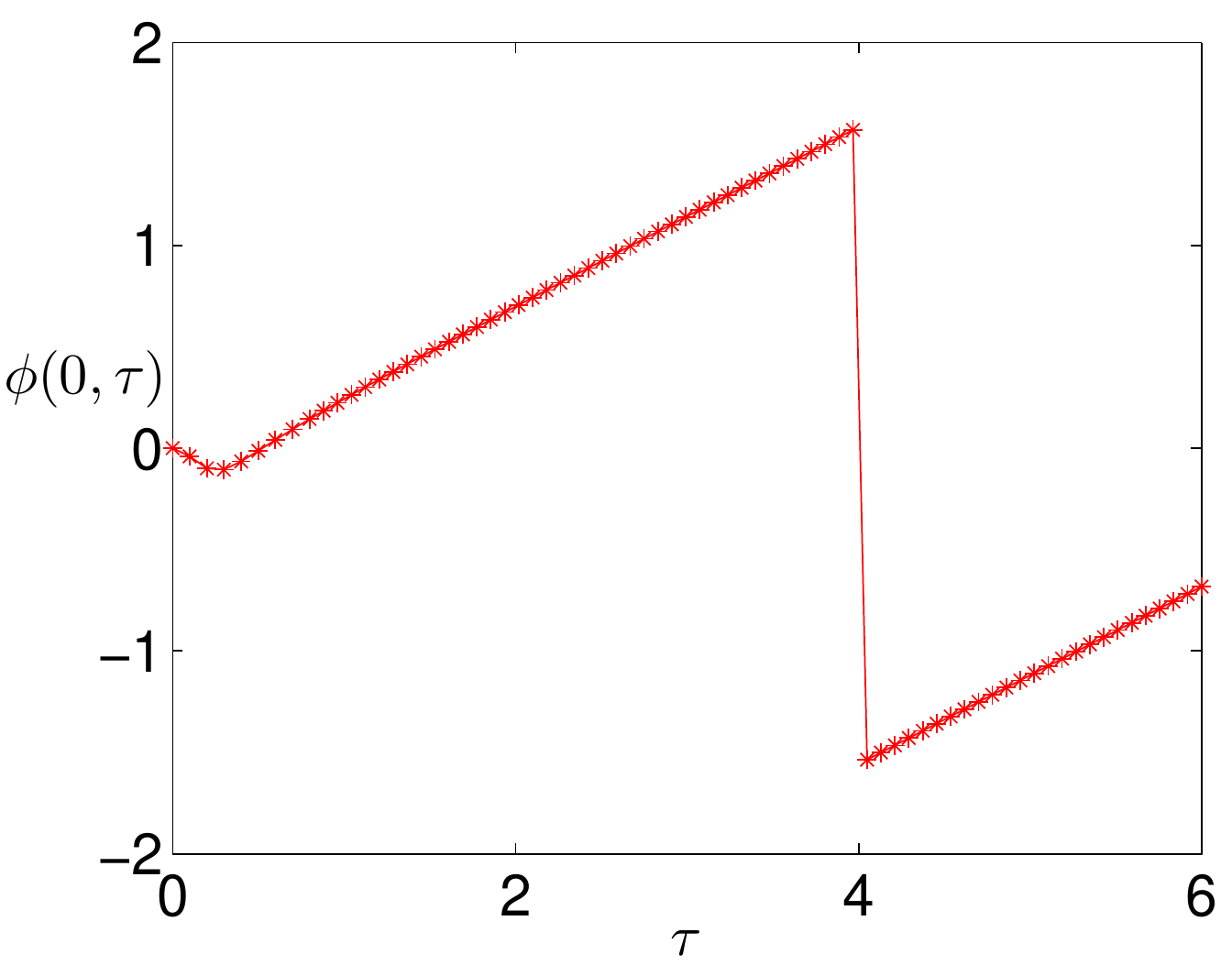}
    \includegraphics[width=0.45\linewidth]{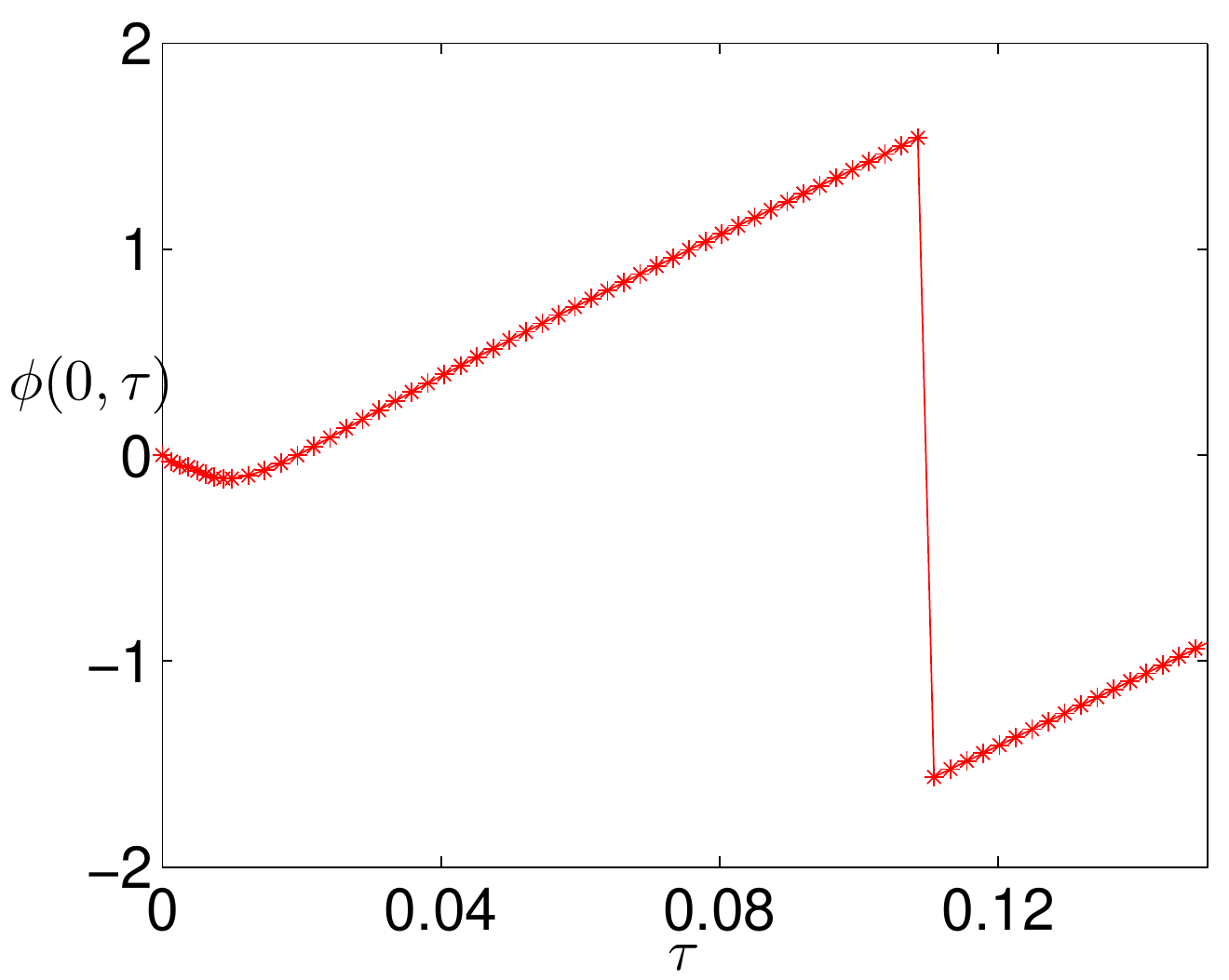}
    \caption{Time evolution of the phase at the origin $\phi(0,\tau)$,
      with initial conditions \eqref{gauss} (left) and \eqref{lor}
      (right).}
    \label{fig:A=3_phase_xi=0}
  \end{center}
\end{figure}

\begin{figure}
  \begin{center}
    \includegraphics[width=0.45\linewidth]{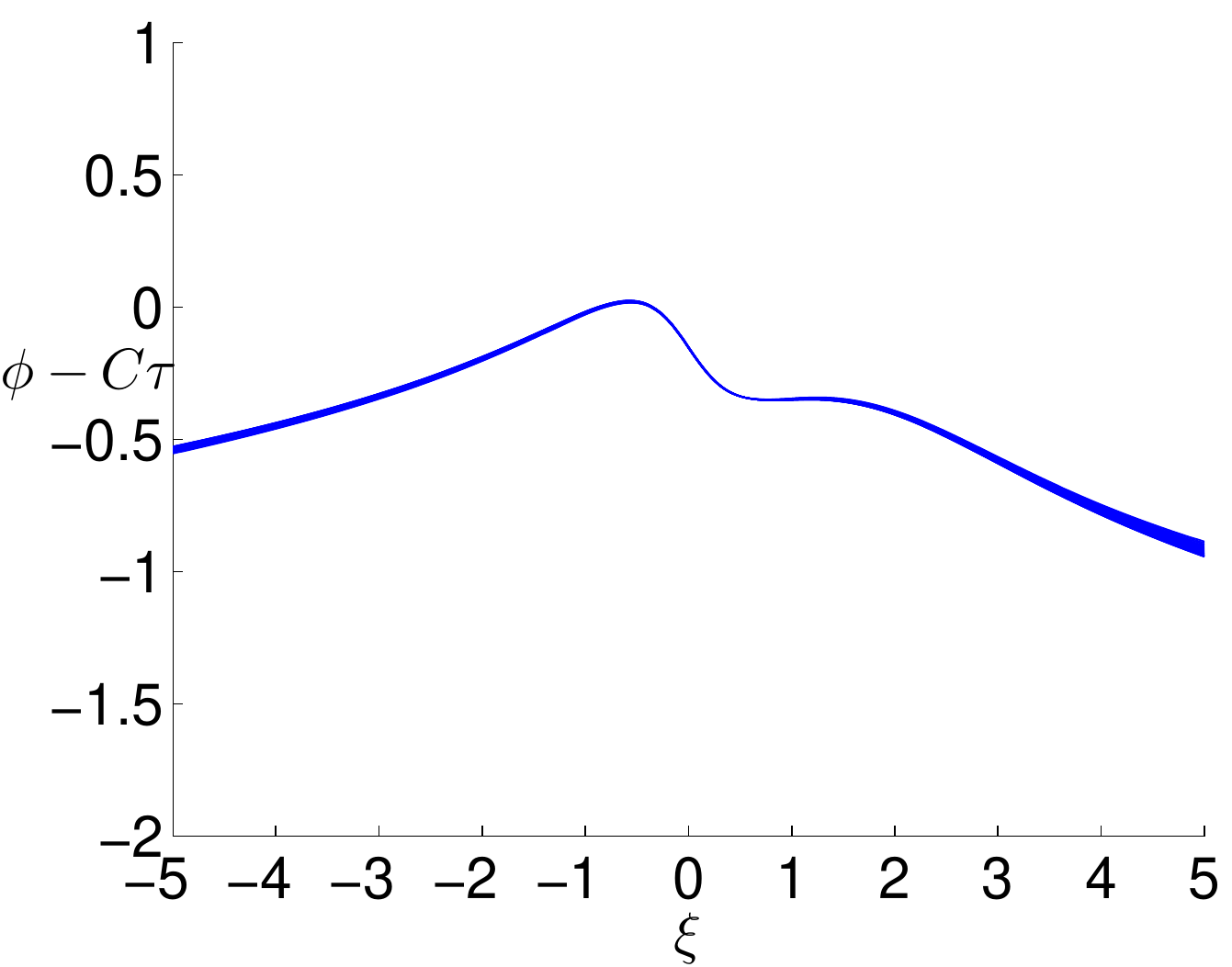}
    \includegraphics[width=0.45\linewidth]{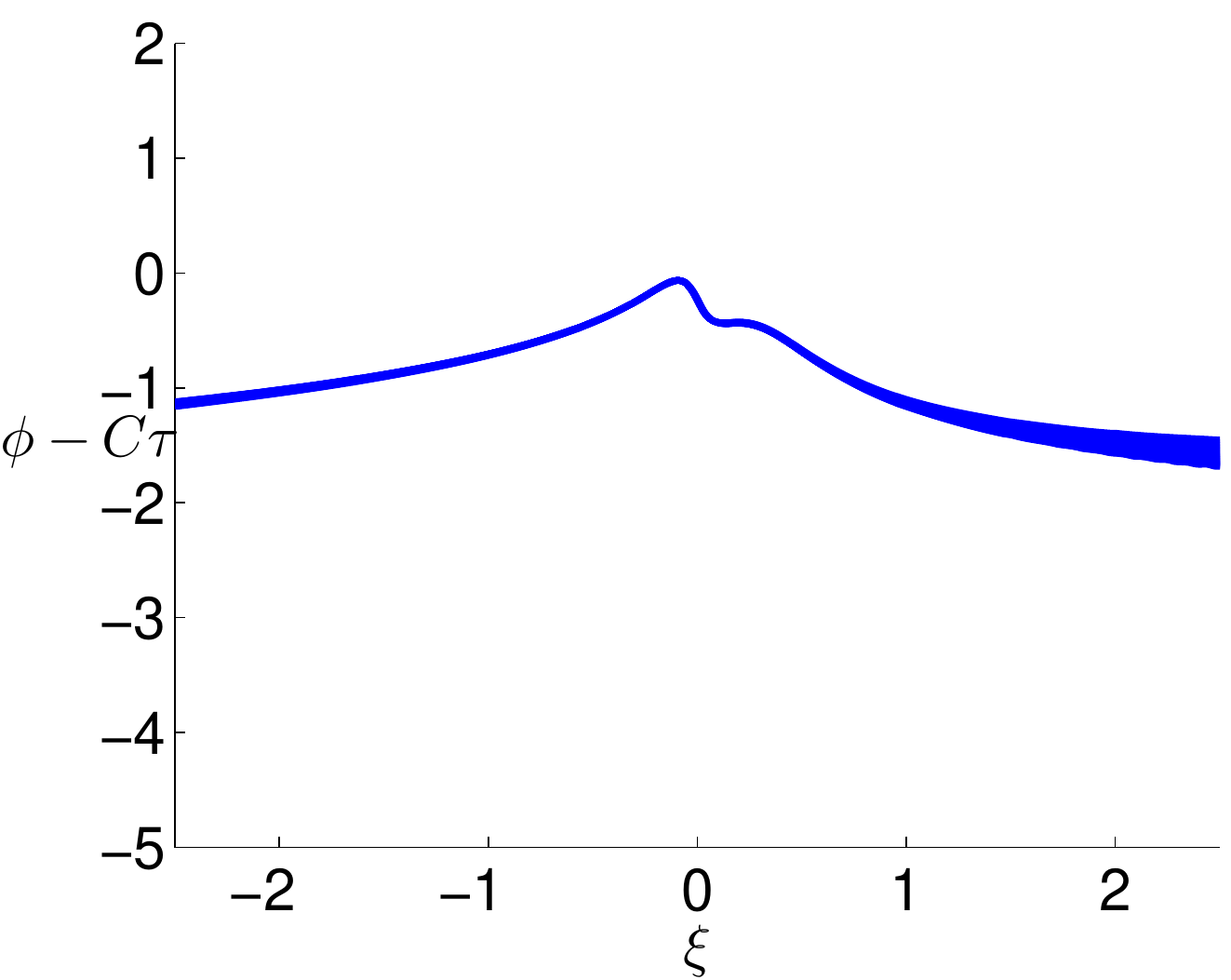}
    \caption{Time evolution of the modified phase
      $\phi(\xi,\tau)-\phi(0,\tau)$, with initial conditions
      \eqref{gauss} for $\tau$ from 4.05 to 6 with an increment of
      0.08.  (left) and \eqref{lor} for $\tau$ from 0.12 to 0.16 with
      an increment of 0.002. (right).}
    \label{fig:A=3_phase}
  \end{center}
\end{figure}

To check the accuracy of our computation, we compute the scaling
factor $L(t)$ in different ways. Either from the invariant quantities
\eqref{eqn:invariants_dyn_res}
\begin{equation}
  \label{eqn:L_em}
  L_{e} = \left(\frac{E(\psi)}{E(u)}\right)^{-\frac{\sigma}{1+\sigma}}, \; \;
  L_{m}     = \left(\frac{M(\psi)}{M(u)}\right)^{-\frac{\sigma}{1-\sigma}}
\end{equation}
or from integration of \eqref{eqn:a_b},
\begin{equation}
  \label{eqn:L_a}
  L_{a}=e^{-\int_0^{\tau} a(\tau)d\tau}.
\end{equation}
Since we have real initial conditions, we cannot use the momentum.
Table \ref{tab:L} shows the values of $L(t)$ obtained in the
simulation corresponding to the initial condition $\psi_0(x) = 3
e^{-2x^2}$ and $\sigma =2$.  The results are in good agreement.

\begin{table}[h]
  \caption{Estimates of $L(\tau)$, calculated using
    \eqref{eqn:L_em} and \eqref{eqn:L_a}, with the initial condition
    \eqref{gauss},  $A_0 = 3$. }
  \label{tab:L}
  \begin{tabular}{c l l l l l l l}
    \hline \hline
    $\tau$ & 1 & 2 & 3 & 4 & 5 & 6\\
    \hline
    $L_e$ & 0.03545&0.015233 & 0.006588& 0.002811 &0.001196  &0.000595\\
    $L_m$ & 0.03545&0.015233 &0.006588 &0.002811 & 0.001196&0.000555\\
    $L_a$ &0.035449 &  0.015233& 0.006588& 0.002811& 0.001196&0.000555\\
    \hline
  \end{tabular}

\end{table}

Using \eqref{eqn:ab} and the fact $a(\tau)$ and $b(\tau)$ tend to
constants A and B, we conclude
\begin{equation}
  \label{eqn:L}
  L^2\sim 2A(t^*-t),
\end{equation}
and,
\begin{equation}
  \frac{dx_0}{dt} =\frac{B}{\sqrt{2A(t^*-t)}},
\end{equation}
where $t^*$ is the blowup time. It follows
\[
x_0\sim x^* - B\sqrt{2(t^*-t)/A},
\]
where $x^*$ is the position of $\max_x{|\psi(x,t^*)|}$. Substituting
into \eqref{eqn:dyn_res_DNLS_system_u_a_b}, we obtain
\begin{equation}
  S_{\xi\xi} - CS +iA(\tfrac{1}{4}S +\xi S_\xi)-iBS_{\xi}+ i|S|^{4}S_\xi = 0.
\end{equation}
Introducing the scaling
\begin{equation}
  \widetilde{\xi} = \sqrt{C}\xi, \quad Q(\widetilde{\xi}) = C^{-\frac{1}{8}} S(\xi),
\end{equation}
and dropping the tildes, we have
\[
Q_{\xi\xi} - Q +i\alpha(\tfrac{1}{4}Q +\xi Q_\xi)-i\beta Q_{\xi}+
i|Q|^{4}Q_\xi = 0,
\]
with the rescaled constants:
\begin{equation}
  \alpha \equiv \frac{A}{C}, \quad \beta\equiv\frac{B}{\sqrt{C}}
\end{equation}
Like supercritical NLS, we find that the coefficients $\alpha$ and
$\beta$ and the function $Q$ are independent of the initial
conditions. Table \ref{tab:a} shows that the ratios $\alpha$ and
$\beta$ take the values $\alpha\sim1.95$ and $\beta\sim2.2$.  In
Figure \ref{fig:Q} and \ref{fig:Arg_Q}, we see that the amplitudes and
the phase of the function $Q$ constructed from the different initial
conditions are coincide.

\begin{table}[h]
  \caption{Limiting values of the  parameters for various initial conditions.}

  \label{tab:a}
  \begin{tabular}{ c c c c c c }
    \hline
    \hline
    $u_0$ & $C$ & $A$ & $B$  &$\alpha$ & $\beta$  \\
    \hline
    $2e^{-2x^2}$  & 1.594 & 3.159 & 2.866& 1.98 & 2.27\\
    $3e^{-2x^2}$ & 0.435&0.854 &1.5 & 1.96 & 2.25  \\
    $4e^{-2x^2}$ &0.098 & 0.187 & 0.691& 1.91  &2.21  \\
    $\frac{3}{(1+(3x)^2)}$ & 16.35 & 31.82 & 8.965& 1.94 & 2.22  \\
    \hline
  \end{tabular}

\end{table}

In conclusion, we have observed that for a large class of initial
conditions, solutions to \eqref{eqn:gDNLS} with quintic nonlinearity
($\sigma =2$) may blow up in a finite time. Their local description
near the singularity point is, up to a rescaling, given by
\eqref{e:universal_blowup}.

The inclusion of the translation parameter is a significant difference
from NLS with power law nonlinearity.  While later studies on
singularity formation in NLS allowed for symmetry breaking
\cite{Landman1991}, the preservation of radial symmetry under the flow
strongly simplifies both the computations and the analysis.  Due to
the mixture of hyperbolic and dispersive terms in gDNLS, no such
symmetry preservation is available, and we must be wary of the
tendency for the solution to migrate rightwards.

\begin{figure}[h]
  \begin{center}
    \includegraphics[width=0.7\linewidth]{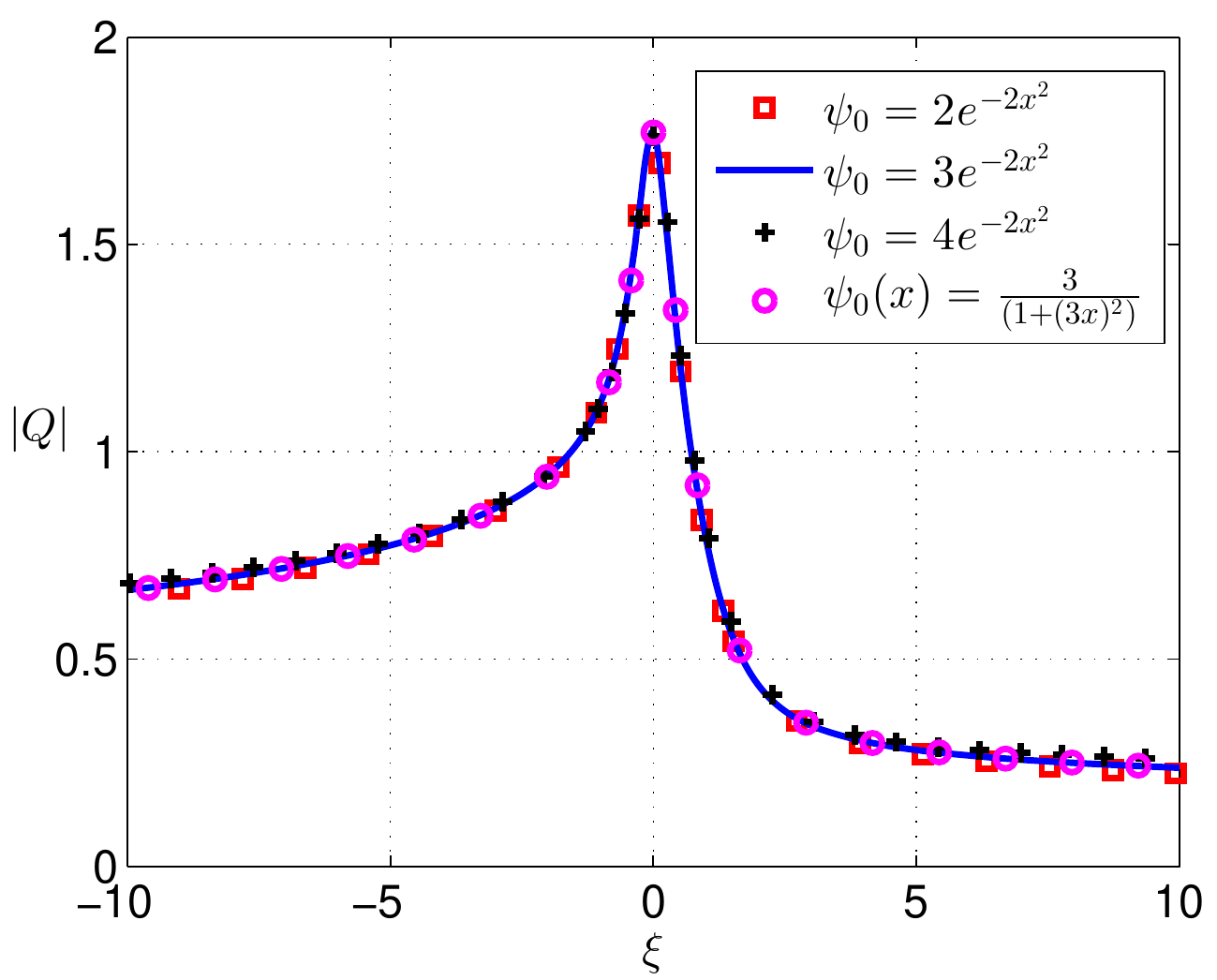}
    \caption{Asymptotic profile $|Q(\xi)|$ for different initial
      data.}
    \label{fig:Q}
  \end{center}
\end{figure}

\begin{figure}[h]
  \begin{center}
    \includegraphics[width=0.7\linewidth]{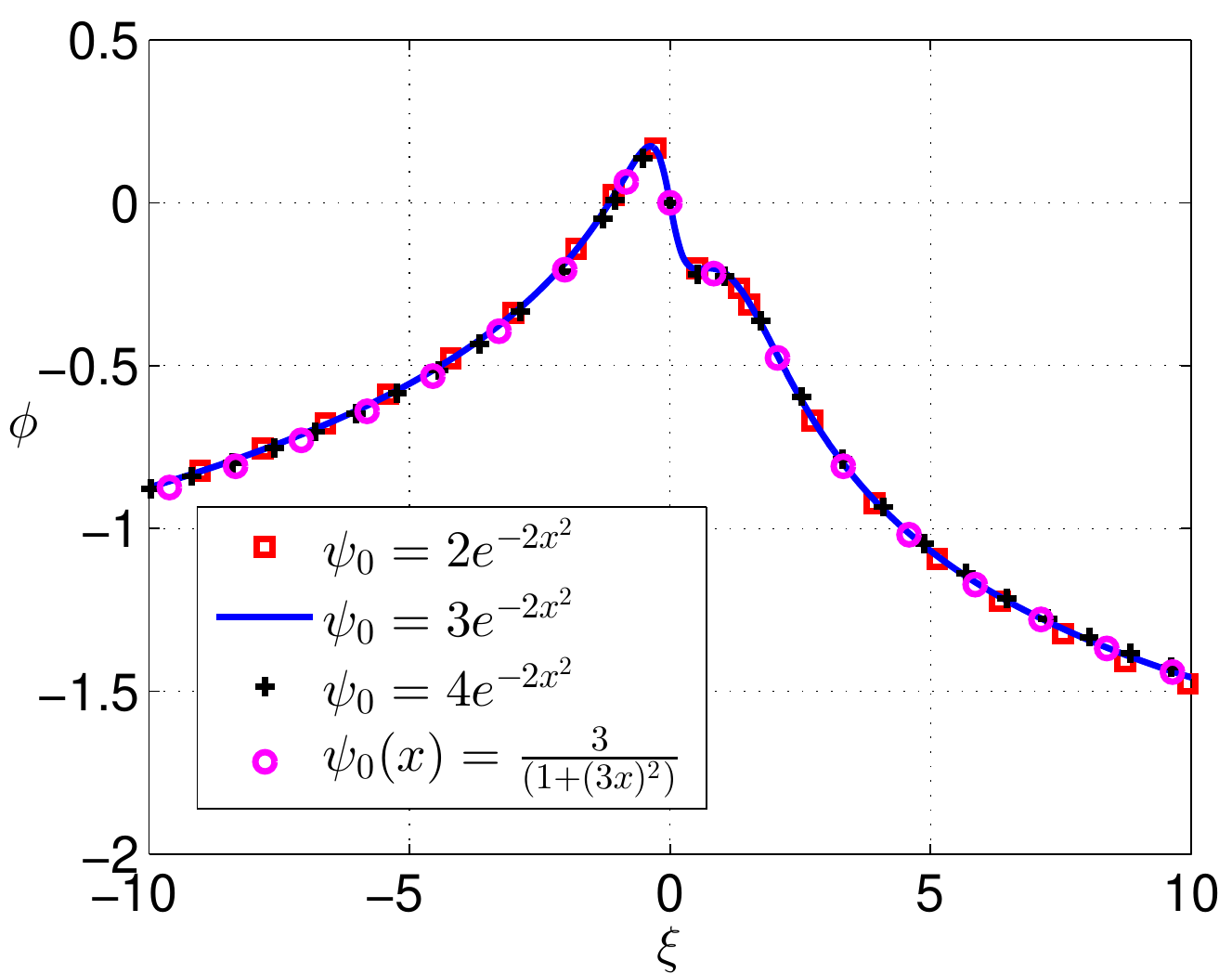}
    \caption{Asymptotic phase $\phi(\xi) \equiv \arg(Q(\xi))$ for
      different initial data.}
    \label{fig:Arg_Q}
  \end{center}
\end{figure}

\section{gDNLS with Other Power Nonlinearities}
\label{s:generalsigma}

The following calculations address the question of singularity for
solutions of \eqref{eqn:gDNLS} when the power in the nonlinear term
$|\psi|^{2\sigma} \psi_x$ is such that $1<\sigma<2$. We recall that
$\sigma=1$ corresponds to the usual DNLS equation which is
$L^2$-critical.  Our calculations cover the values of $\sigma$ down to
1.1.  Computational difficulties precluded us from reducing it much
further.

Figure \ref{fig:a_b_sigma_neq_2} shows that $a(\tau)$ and $b(\tau)$
tend to constant values as $\tau$ increases.  For convenience, we have
plotted $a(\tau)/a_M$ and $b(\tau)/b_M$ versus $\tau/\tau_M$, where
$\tau_M$ is the maximum time of integration and $a_M= a(\tau_M)$,
$b_M=b(\tau_M$).  The initial condition is $\psi_0(x) = 3e^{-2x^2}$
when $\sigma = 1.7, 1.5$ and 1.3, while $\psi_0 = 4e^{-2x^2}$ for
$\sigma =1.1$.  We also remark that the more supercritical, the larger
$\sigma$ is, the shorter the transient period.  Table
\ref{tab:a_sigma} shows that $\alpha$ decreases as $\sigma$ varies
form 2 to 1.1, while $\beta$ first decreases and then increases.

The main conclusion of our study is that for this range of values of
$\sigma$, solutions to \eqref{eqn:gDNLS} may blow up in a finite time
and their local structure is similar to that of the case $\sigma=2$
discussed in the previous section and is given by
\eqref{e:universal_blowup}.

\begin{table}[h]
  \caption{Values of $\alpha$ and $\beta$ as $\sigma$ approaches to
    1.}

  \label{tab:a_sigma}
  \begin{tabular}{c c c c c c}
    \hline
    \hline
    $\sigma$ & 2 &   1.7  & 1.5 & 1.3 & 1.1\\
    \hline
    $\alpha$        & 1.98 & 1.26 & 0.85 & 0.38 & 0.04\\
    $\beta$        & 2.27 & 2.11 & 1.68 & 1.63 & 1.90\\
    \hline
  \end{tabular}
\end{table}

\begin{figure}
  \includegraphics[width=0.8\linewidth,type=pdf,ext=.pdf,read=.pdf]{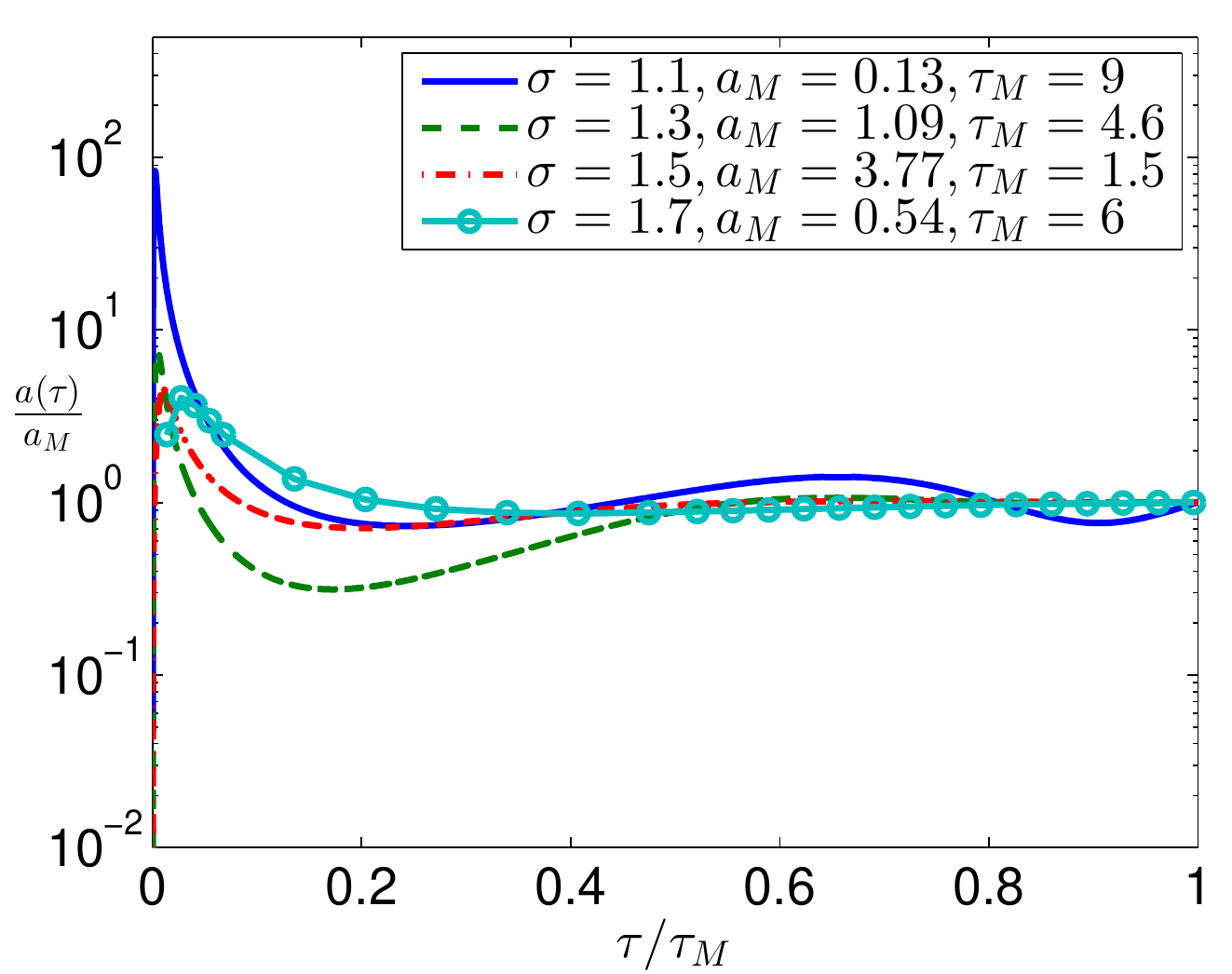}
  \includegraphics[width=0.8\linewidth,type=pdf,ext=.pdf,read=.pdf]{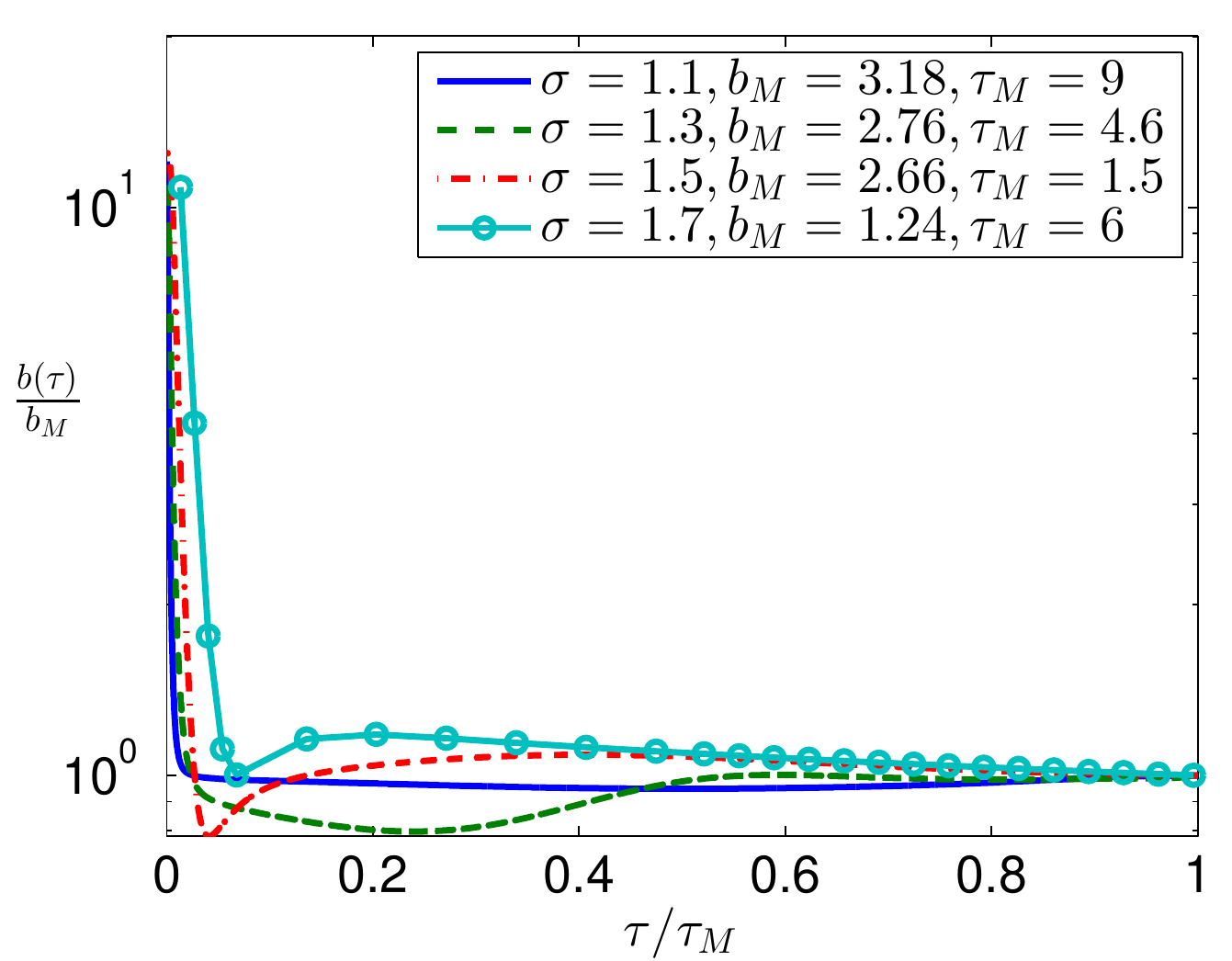}
  \caption{ $a(\tau)$ (up) and $b(\tau)$ (down) versus $\tau$, after
    normalization by $a_M = a(\tau_M)$ and $b_M = b(\tau_M)$, $\tau_M
    = \max(\tau)$ .}
  \label{fig:a_b_sigma_neq_2}
\end{figure}

\section{Blowup Profile}
\label{s:profile}
This section is devoted to the study of the nonlinear elliptic
equation \eqref{e:blowup_soln} satisfied by the complex profile
function $Q$.  It is helpful to change variables, letting $\eta
=\xi-\frac{\beta}{\alpha}$, leading to
\begin{equation}
  \label{eqn:Q_profile_2}
  Q_{\eta\eta} - Q +i\alpha(\tfrac{1}{2\sigma}Q +\eta Q_\eta)+
  i|Q|^{2\sigma}Q_\eta = 0.
\end{equation}
We seek profiles $|Q(\eta)|$ that decrease monotonically with
$|\eta|$, satisfying the conditions
\begin{equation*}
  Q(\eta)\rightarrow  0\quad\mbox{as}\quad \eta\rightarrow \pm\infty.
\end{equation*}
This can be viewed as a nonlinear eigenvalue problem, with
eigenparameter $\alpha$, and eigenfunction $Q$.  As the equation is
invariant to multiplication by a constant phase, we can assume, in
addition, that $Q(0)$ is real (for example).

\subsection{Properties of the Asymptotic Profile}

\begin{proposition}
  \label{prop:large_dis_beh}
  The asymptotic behavior of solutions to \eqref{eqn:Q_profile_2} as
  $\eta \to \pm \infty$, is given by $Q=c_1^\pm Q_1+c_2^\pm Q_2$ where
  \begin{equation}
    Q_1 \approx
    |\eta|^{-\frac{i}{\alpha}  -\frac{1}{2\sigma}} , \; \; \;
    Q_{2} \approx  e^{i(-\frac{\alpha}{2} \eta^2  )}
    |\eta| ^{\frac{i}{\alpha}+  \frac{1}{2\sigma}-1},
  \end{equation}
  and $c_1^\pm,c_2^\pm$ are complex numbers.
\end{proposition}
\begin{proof}
  We first write $Q(\eta)=X(\eta)Z(\eta)$ and choose $X$ so that the
  resulting second-order equation for $Z$ does not contain first order
  terms.  After subsitution in the $Q$-equation, we get
  \begin{equation}
    \begin{split}
      X Z_{\eta\eta} + &(2 X_\eta + i \alpha \eta X + iX|X|^{2\sigma}|Z|^{2\sigma}) Z_\eta\\
      + &(X_{\eta\eta} - X + i \frac{\alpha}{2\sigma} X + i \alpha
      \eta X_\eta + i|X|^{2\sigma}|Z|^{2\sigma} X_\eta ) Z = 0.
    \end{split}
  \end{equation}
  We now let
  \begin{equation}
    \label{eqn:X}
    \begin{split}
      X &= \exp\set{-i \frac{\alpha}{4} \eta^2 -
        i\frac{1}{2}\int_0^\eta |Z|^{2\sigma} d\eta } \\
      &=\exp\set{-i \frac{\alpha}{4} \eta^2 - i\frac{1}{2}\int_0^\eta
        |Q|^{2\sigma} d\eta },
    \end{split}
  \end{equation}
  so that the equation for $Z$ reduces to
  \begin{equation}
    \label{eqn:Z_full}
    \begin{split}
      Z_{\eta\eta} + &\left(\frac{i\alpha(1-\sigma)}{2\sigma} -1 +
        \frac{(\alpha \eta)^2}{4} + \frac{|Z|^{4\sigma}}{4} +
        \frac{\alpha}{2}\eta|Z|^{2\sigma}
        -\frac{i}{2}(|Z|^{2\sigma})_\eta\right) Z=0.
    \end{split}
  \end{equation}
  Denoting $Z= Ae^{i\varphi}$ and $\theta= \varphi_\eta$, we
  rewrite\eqref{eqn:Z_full} in terms of phase and amplitude:

  \begin{subequations}
    \begin{equation}
      \label{eqn:A}
      \frac{A''}{A} -\theta^2 -1 + \frac{\alpha^2\eta^2}{4} +\frac{\alpha\eta}{2}A^{2\sigma} + \frac{1}{4} A^{4\sigma} =0
    \end{equation}
    \begin{equation}
      \label{eqn:theta}
      \theta' + 2\theta\frac{A'}{A} + \frac{\alpha(1-\sigma)}{2\sigma} -\frac{1}{2}(A^{2\sigma})' = 0
    \end{equation}
  \end{subequations}

  Anticipating that $A \to 0$ when $\eta\to \pm\infty$, we see that to
  balance the quadratic and constant in $\eta$ terms of \eqref{eqn:A},
  we must have
  \begin{equation}
    \theta_1 = \left ( \frac{\alpha \eta}{2} - \frac{1}{\alpha \eta} \right) +\gamma_1(\eta),\quad
    \theta_2 =  -\left( \frac{\alpha \eta}{2} - \frac{1}{\alpha \eta} \right) +\gamma_2(\eta)
  \end{equation}
  where $\gamma_{1,2}(\eta)$ is at most $O(\eta^{-1})$.  This leads to
  two cases:

  \noindent {\bf Case 1.} \ $ \theta_1 = ( \frac{\alpha \eta}{2} -
  \frac{1}{\alpha \eta} ) +\gamma_1(\eta)$. Assuming $ A_1\approx
  c_1^{\pm}|\eta|^{-p}$,where $c_1$ is a constant, we get from
  \eqref{eqn:theta} that $p=1/2\sigma$ by balancing the $O(1)$ terms.

  Substituting these estimates of $\theta_1$ and $A_1$ into
  \eqref{eqn:A}, we have to balance the $O(1)$ terms $\alpha \eta
  \gamma_1$ and $\frac{\alpha}{2} \eta A_1^{2\sigma}$.  Hence, as
  $\eta\to \pm \infty$,
  \begin{equation}
    \theta_1 \approx   \frac{\alpha \eta}{2}
    - \frac{1}{\alpha \eta}  + \frac{1}{2} |Q_1|^{2\sigma} + o(\eta^{-1}); \; \;
    A_1\approx c_1^{\pm} |\eta|^{-1/{2\sigma}}(1+o(1)).
  \end{equation}
  Retaining $\gamma_1$ is essential since it is of order
  $O(\eta^{-1})$ and is the only term capable of balancing $\eta
  A_1^{2\sigma}$.  Returning to the function $Q_1$, we get one family
  of solutions with the farfield behavior
  \begin{equation}
    Q_1 \approx c_1^{\pm}\abs{\eta}^{-\frac{1}{2\sigma} -\frac{i}{\alpha}} \; {\rm as} \;
    \eta \to \pm \infty.
  \end{equation}

  \noindent {\bf Case 2.}  \ $ \theta_2 = - ( \frac{\alpha \eta}{2} -
  \frac{1}{\alpha \eta} ) +\gamma_2(\eta)$ . Assuming, again, $A_2$ is
  asymptotically a power law, we get from \eqref{eqn:theta},
  $A_2\approx c_2^{\pm} |\eta|^{-1+ \frac{1}{2\sigma}}$, where
  $c_2^{\pm}$ is a constant.

  Substituting back into \eqref{eqn:A}\, we get $\gamma_2\approx
  -\tfrac{1}{2}\abs{A_2}^{2\sigma}$, so that
  \begin{equation}
    \theta_2 \approx -\frac{\alpha \eta}{2}
    +\frac{1}{\alpha \eta} - \frac{1}{2} \abs{Q_2}^{2\sigma} +
    o(\eta^{1-2\sigma}); \;\; A_2 \approx c_2^{\pm} \abs{\eta}^{-1+ \frac{1}{2\sigma}}(1+o(1))
  \end{equation}
  leading to
  \begin{equation}
    Q_2\approx c_2^{\pm}  e^{-i\frac{\alpha}{2} \eta^2} |\eta|^{\frac{i}{\alpha} +
      \frac{1}{2\sigma}-1} e^{- i \int_0^\eta |Q_2|^{2\sigma} d\eta}.
  \end{equation}
  Finally, we can neglect the last term in the phase and write
  \begin{equation}
    Q_2\approx  c_2^{\pm} e^{-i\frac{\alpha}{2} \eta^2} |\eta|^{\frac{i}{\alpha} +
      \frac{1}{2\sigma}-1} \; {\rm as} \;
    \eta \to \pm \infty.
  \end{equation}

\end{proof}
\begin{proposition}
  If $Q$ is a solution of \eqref{eqn:Q_profile_2} with $Q_\eta \in
  L^2(\mathbb{R})$ and $Q \in L^{4\sigma +2}(\mathbb{R})$, its energy
  vanishes:
  \begin{equation}
    \label{eqn:ham_van}
    \int_{-\infty}^\infty|Q_\eta|^2 d\eta+\frac{1}{(\sigma+1)}\Im\int_{-\infty}^\infty |Q|^{2\sigma}\bar{Q}Q_\eta d\eta=0.
  \end{equation}
  \label{p:ham_van}
\end{proposition}
\begin{proof}
  Multiplying \eqref{eqn:Q_profile_2} by $\bar{Q}_{\eta\eta}$, taking
  the imaginary parts, and integrating over the whole line gives
  \begin{equation}
    \label{eqn:Q_profile_im}
    \alpha\Re\int\left(\frac{Q}{2\sigma} +\eta Q_{\eta}\right) \bar{Q}_{\eta\eta}d\eta +\Re\int |Q|^{2\sigma} Q_\eta \bar{Q}_{\eta\eta} d\eta = 0.
  \end{equation}
  The first integral of \eqref{eqn:Q_profile_im} can be written
  \begin{equation}\label{4.4}
    \begin{split}
      \alpha\Re\int\left(\frac{Q}{2\sigma} +\eta Q_{\eta}\right)
      \bar{Q}_{\eta\eta}d\eta
      =\left(-\frac{\sigma+1}{2\sigma}\right)\alpha \int |Q_\eta|^2
      d\eta.
    \end{split}
  \end{equation}
  Using \eqref{eqn:Q_profile_2} to express $\bar{Q}_{\eta\eta}$, the
  second integral becomes
  \begin{equation} \label{4.5}
    \begin{split}
      &\Re \int |Q|^{2\sigma} Q_\eta (\bar{Q} +i\alpha(\frac{1}{2\sigma} \bar{Q} +\eta\bar{Q}_\eta) i|Q|^{2\sigma} \bar{Q}_\xi) d\eta \\
      & =\Re \int |Q|^{2\sigma}Q_\eta\bar{Q}
      +\frac{i\alpha}{2\sigma}|Q|^{2\sigma}Q_\eta\bar{Q}d\eta\\
      &= -\frac{\alpha}{2\sigma}\Im\int
      |Q|^{2\sigma}Q_{\eta}\bar{Q}d\eta.
    \end{split}
  \end{equation}
  Combining \eqref{4.4} and \eqref{4.5}, we obtain
  \eqref{eqn:ham_van}.
\end{proof}
\begin{proposition}
  If $Q$ is a solution of \eqref{eqn:Q_profile_2} with $\sigma>1$ and
  $\alpha >0$, and $Q\in H^1(\mathbb{R}) \bigcap
  L^{2\sigma+2}(\mathbb{R})$, then $Q \equiv 0$.
\end{proposition}
\begin{proof}
  Multiplying \eqref{eqn:Q_profile_2} by $\bar{Q}$, taking the
  imaginary part, and integrating over $\mathbb{R}$, we get
  \begin{equation}
    \frac{\alpha}{2} \left(\frac{1}{\sigma} -1\right)\int |Q|^2 d\eta = 0,
  \end{equation}
  thus $Q$ has to be identically zero.
\end{proof}
Consequently, solutions with finite energy have an infinite
$L^2$-norm.

\subsection{Numerical Integration of the Boundary Value Problem}

The purpose of this section is the numerical integration of the BVP
\eqref{eqn:Q_profile_2} in order to better understand the asymptotic
profile of the singular solutions of gDNLS.  We look for solutions
that behave like $c_1^{\pm} Q_1$ as $|\eta| \to\infty$,where
$c_1^{\pm}$ are complex constants, because $Q_2 $ does not have finite
energy.  It is convenient to rewrite the large $\eta$ behavior as a
Robin boundary condition of the form
\begin{equation}
  \label{BC-Q1}
  - Q + i \alpha \left(\frac{1}{2\sigma} Q + \eta Q_\eta \right) =0, \quad  |\eta|\to \infty.
\end{equation}
Solutions of \eqref{eqn:Q_profile_2} depend on the coefficient
$\alpha$.  Since the equation is invariant under phase translation, we
need an additional condition; for example, setting the phase to zero
at a particular point is satisfactory. This suggests that $\alpha$
needs to take particular values, like in the supercritical NLS
problem, \cite{Sulem1999}.  We are particularly interested in the
character of the asymptotic profile and of the coefficient $\alpha$ as
$\sigma $ approaches the critical case $\sigma=1$.  Our basic approach
is a continuation method in the parameter $\sigma$.

At first, we attempted to integrate \eqref{eqn:Q_profile_2} with
boundary conditions \eqref{BC-Q1} for values of $\sigma$ starting from
$\sigma=2$ to about $\sigma=1.2$. We observed that the peak of $|Q|$
rapidly moved to the left of the domain as $\sigma$ decreased,
limiting our calculations, (Figure \ref{fig:continuation_Q_moving}).
\begin{figure}
  \centering
  \includegraphics[width=0.8\linewidth]{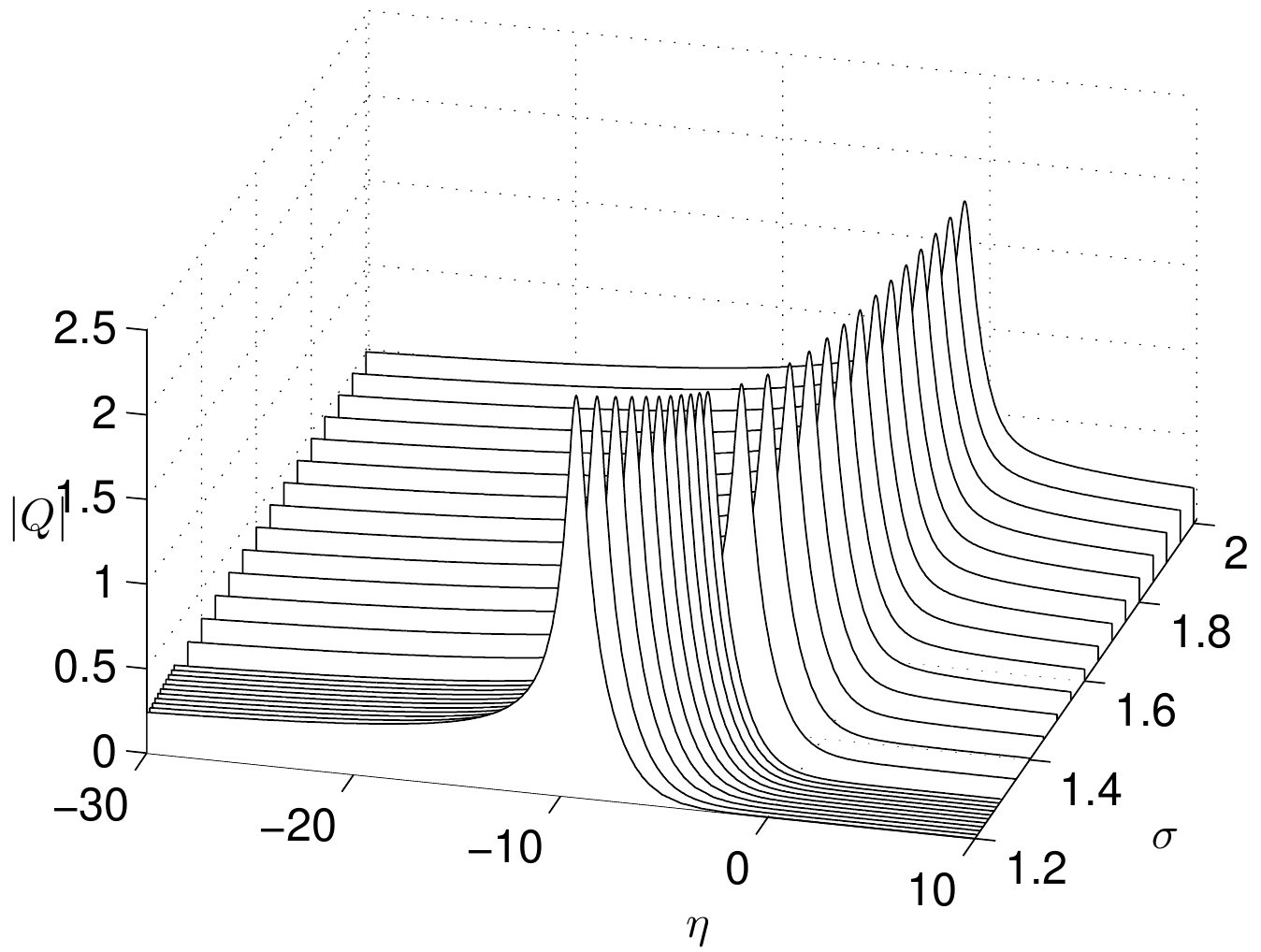}\\
  \caption{Asymptotic profile $|Q(\eta)|$ for various $\sigma$, where
    $\beta$ has been translated to zero and the peak is permitted
    move.}
  \label{fig:continuation_Q_moving}
\end{figure}

To reach values of $\sigma $ closer to $1$, we returned to
\eqref{e:blowup_soln} (which is equivalent to \eqref{eqn:Q_profile_2}
after the translation $\xi =\eta+ \beta/\alpha$). Now the parameter
$\beta$ is free and will be chosen so that the maximum of $|Q|$ is at
$\xi =0$.  This adds a condition, namely $|Q|_\xi(0)=0$ and an
additional unknown, the coefficient $\beta$.  Figure
\ref{fig:continuation_Q} shows that the solution of
\eqref{e:blowup_soln} when $\sigma$ varies from 2 to 1.08.  As $\sigma
\to 1$, we make several observations on $Q$.  The amplitude increases,
and the left shoulder becomes lower.  Oscillations also appear in the
real and imaginary components (Figure \ref{fig:continuation_ReQ_ImQ}).
We observe that the parameter $\alpha$ decreases as $\sigma$
approaches 1, while $\beta$ first decreases and then increases (Figure
\ref{fig:bvp_a_b}).  A detailed analysis of the dependence of these
parameters on $\sigma$ is the subject of a future study.

As a practical matter, we integrated \eqref{e:blowup_soln} in two
adjacent domains $(-\infty, 0]$ and $[0, \infty)$, using the
multipoint feature of the {\sc Matlab} {\tt{bvp4c}} solver, with Robin
boundary conditions at $\pm \infty$ and continuity conditions on $Q$
and its first derivative $Q_\xi$ at $\xi =0$.  Additional details are
given in Appendix \ref{sec:BVP}.

\begin{figure}[h]
  \begin{center}
    \includegraphics[width=0.8\linewidth]{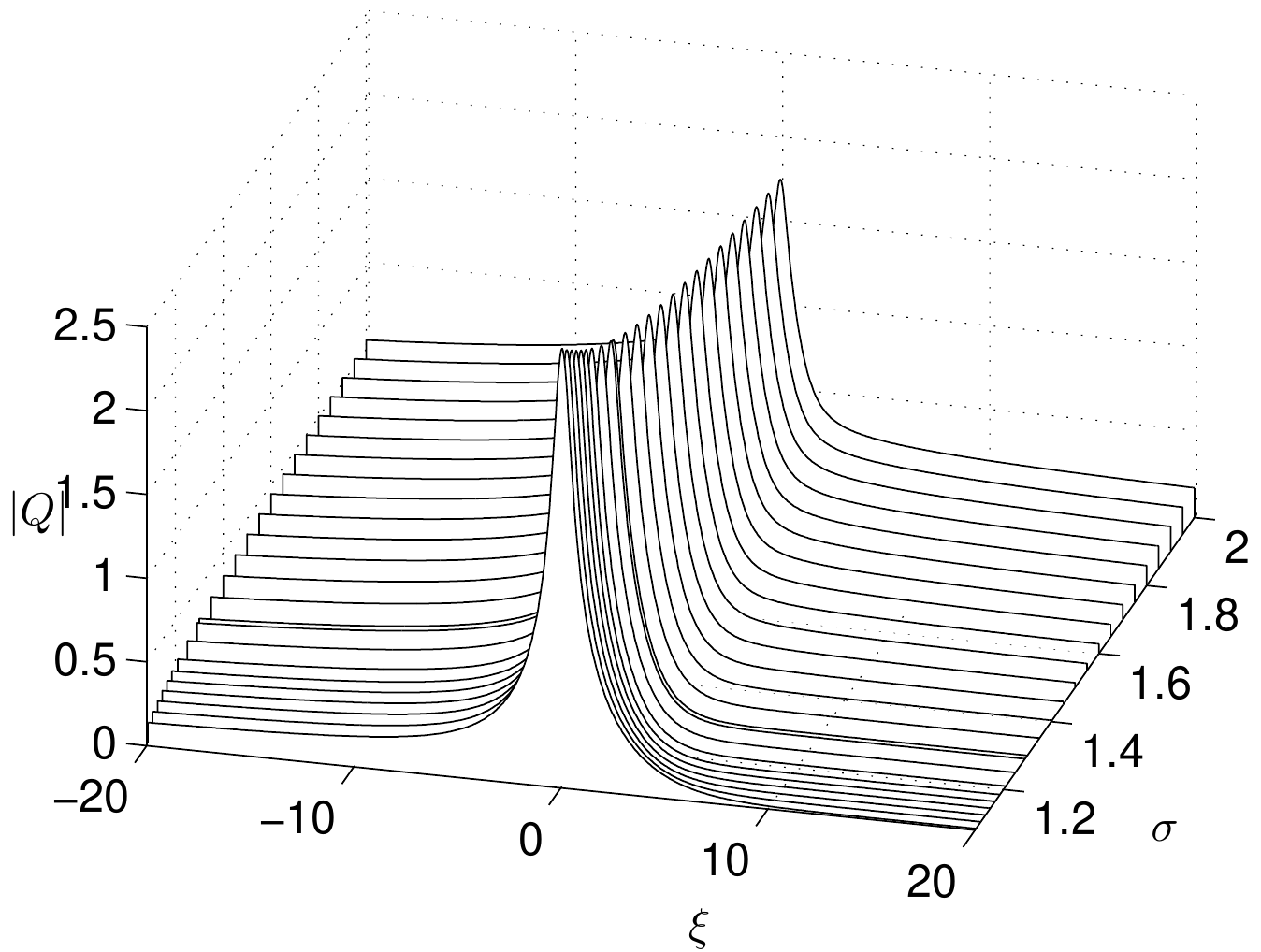}
    \caption{Asymptotic profile $|Q(\xi)|$ for various $\sigma$, where
      $\beta$ is a free parameter and the peak is fixed at the
      origin.}
    \label{fig:continuation_Q}
  \end{center}
\end{figure}

\begin{figure}[h]
  \begin{center}
    \includegraphics[width=0.45\linewidth]{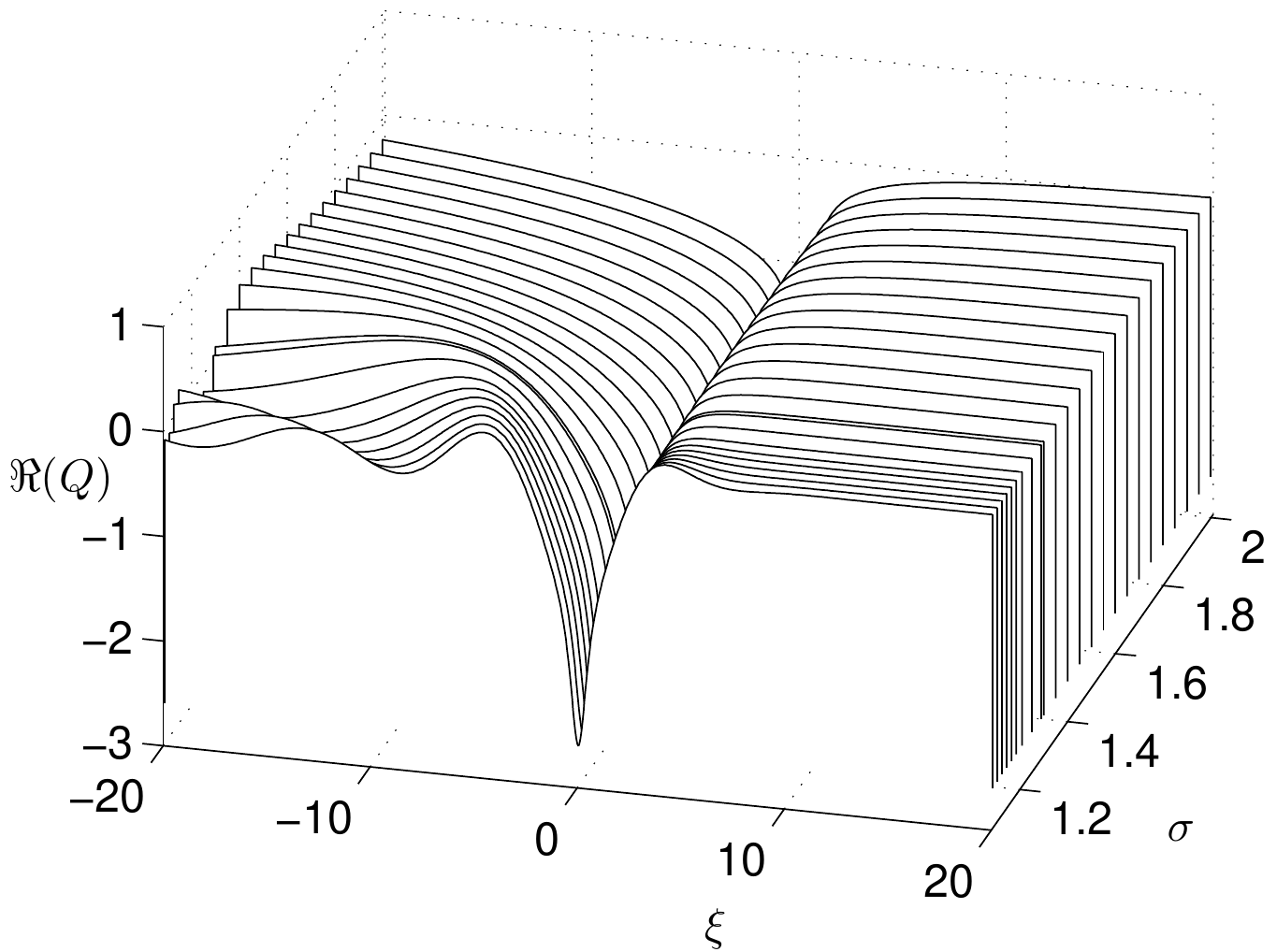}
    \includegraphics[width=0.45\linewidth]{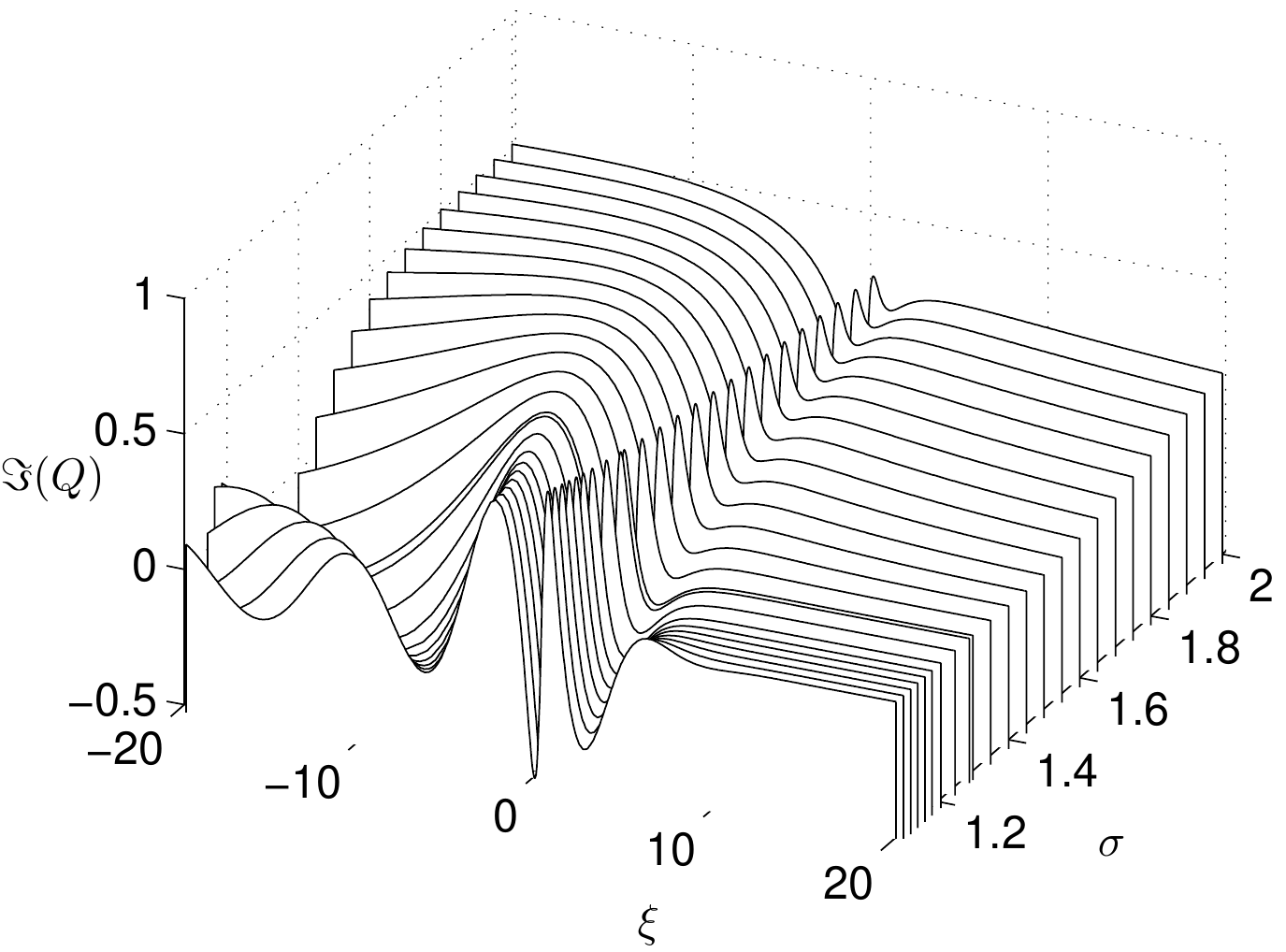}
    \caption{Asymptotic profile $\Re(Q)$ (left) and $\Im(Q)$ (right)
      for various $\sigma$.}
    \label{fig:continuation_ReQ_ImQ}
  \end{center}
\end{figure}

\begin{figure}[h]
  \begin{center}
    \includegraphics[width=0.45\linewidth]{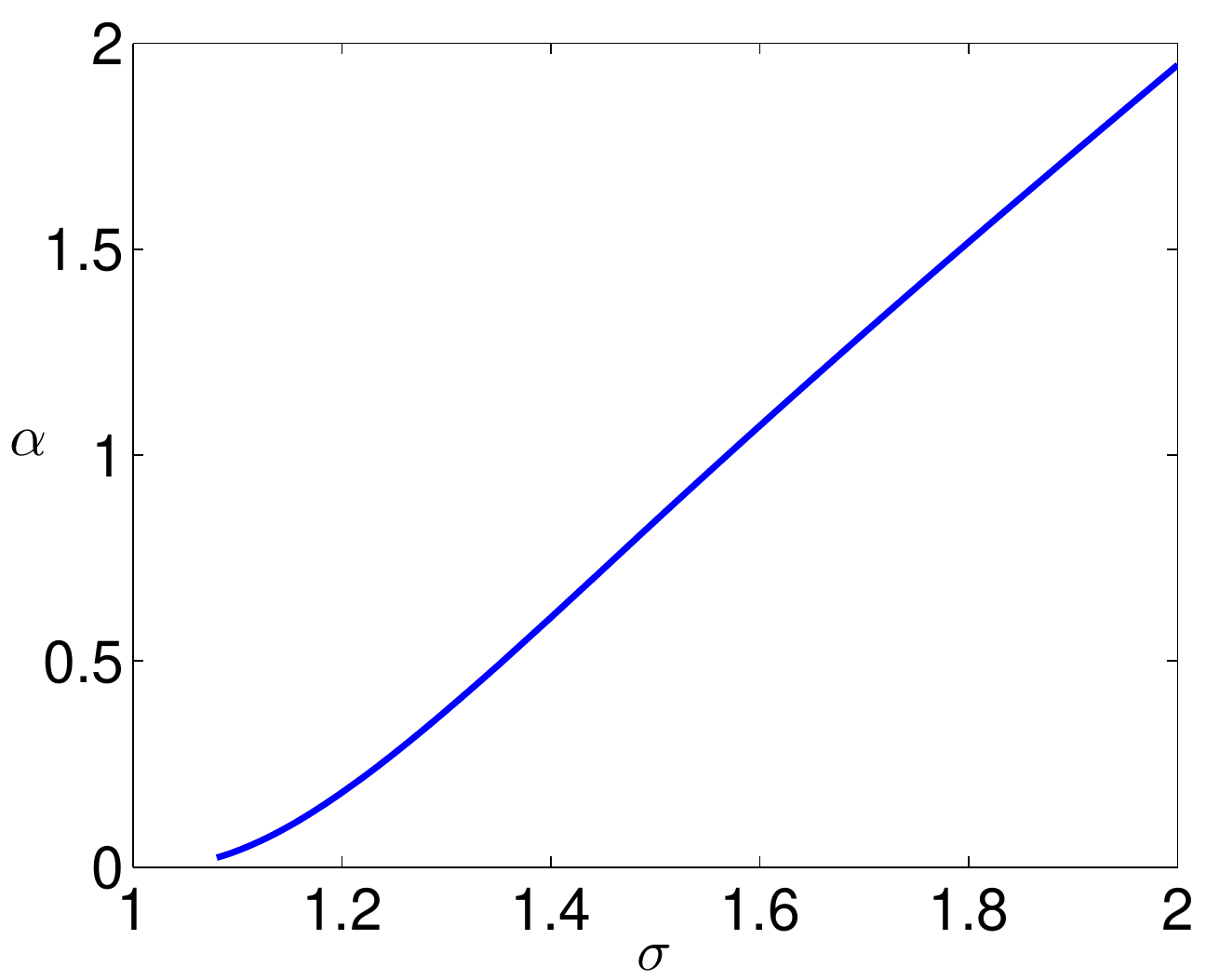}
    \includegraphics[width=0.45\linewidth]{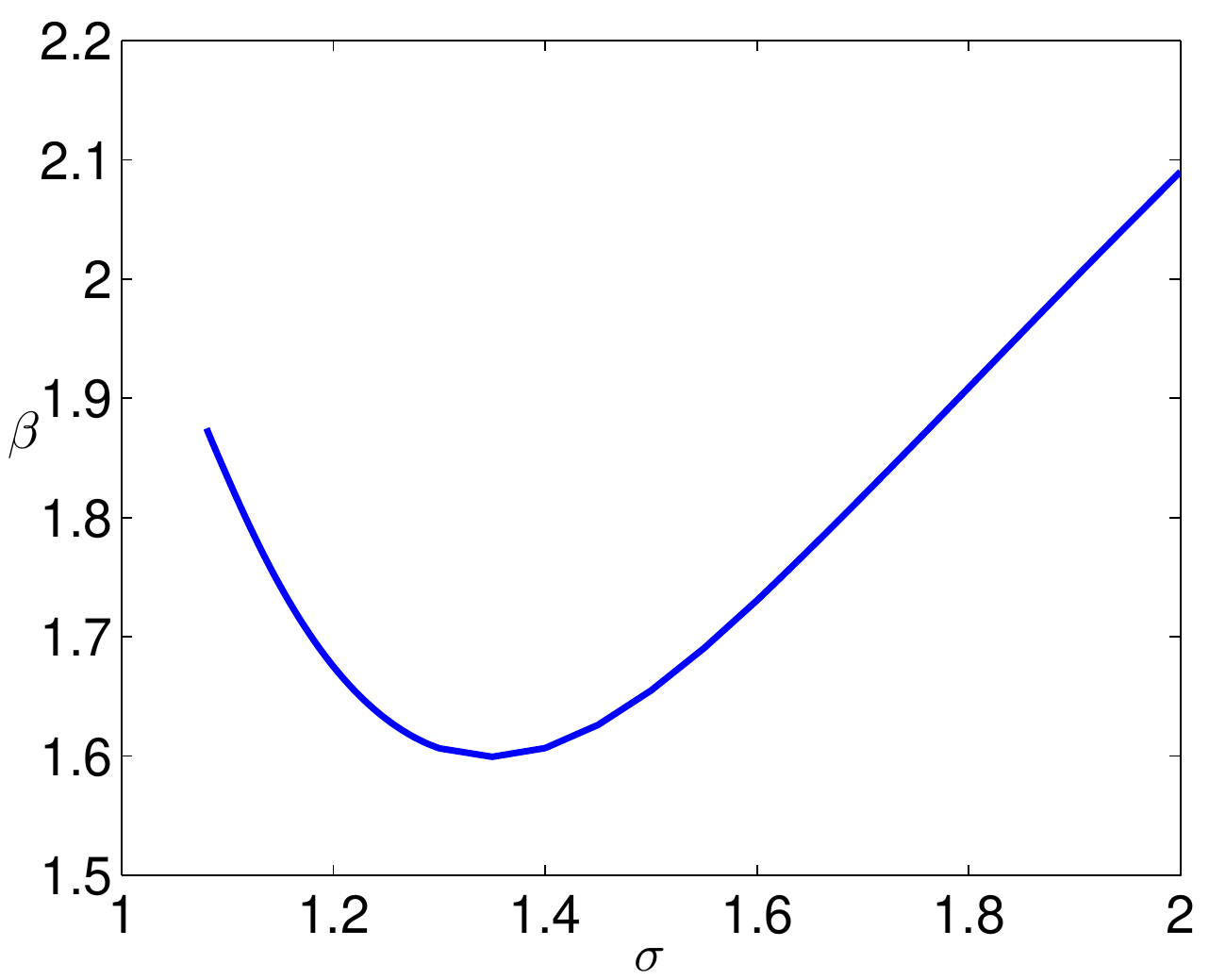}
    \caption{Coefficients $\alpha$ (left) and $\beta$ (right) versus
      $\sigma$.}
    \label{fig:bvp_a_b}
  \end{center}
\end{figure}

\section{Discussion}
\label{s:discussion}

We have numerically solved a derivative NLS equation with a general
power nonlinearity and found evidence of a finite time singularity.
We have determined that there is a square root blowup rate for the
scaling factor $L(t)$.  This implies that Sobolev norms should grow as
\[
\|\partial_x^s\psi(x,t)\|_{L^2} \sim (
t^*-t)^{\frac{\sigma-1}{4\sigma} -\frac{s}{2}}.
\]

As with supercritical NLS, this equation has a universal blowup
profile of the form \eqref{e:universal_blowup}.  Indeed, $Q$, $\alpha$
and $\beta$ are solutions of a nonlinear eigenvalue problem, and they
do not depend on the initial conditions. They depend on the power
nonlinearity $\sigma$.  Another similarity to supercritical NLS is
that the blowup profile, $Q$, is not in $L^2$ and has zero energy.
Note that there are other types of singular solutions to supercritical
NLS that were found recently \cite{FGW07} \cite{HR07}.

To conclude, we present numerical simulations of the DNLS equation
with $\sigma =1$ and several Gaussian initial conditions ($A_0 =
3,4,5,6$).  In all our simulations, the solution separates into
several waves and disperses. None of our simulations have shown any
evidence of a finite time singularity, although the transient dynamics
can be quite violent.  Figure \ref{fig:A=3_water_fall_sigma=1} shows
the evolution of $|\psi(x,t)|$ with the Gaussian initial condition:
$\psi_0(x) = 6 e^{-2 x^2}$.  The maximum value of $|\psi|$ increases
slowly in time and then stabilizes.  We also integrated DNLS using the
dynamic rescaling method.  We found that $a(\tau)$ rapidly tends to
zero and that the scaling factor $L(t)$ has a lower bound away from
zero, (Figure \ref{fig:sigma=1_a}).  This is different with the
critical NLS, which has a blowup rate at $\{ \ln \ln [(t^*
-t)^{-1}]/(t^*-t)\}^{1/2} $, due to the slow decay rate of $a(\tau)$,
\cite{Landman1988}.
\begin{figure}[h]
  \begin{center}
    \includegraphics[width=0.7\linewidth]{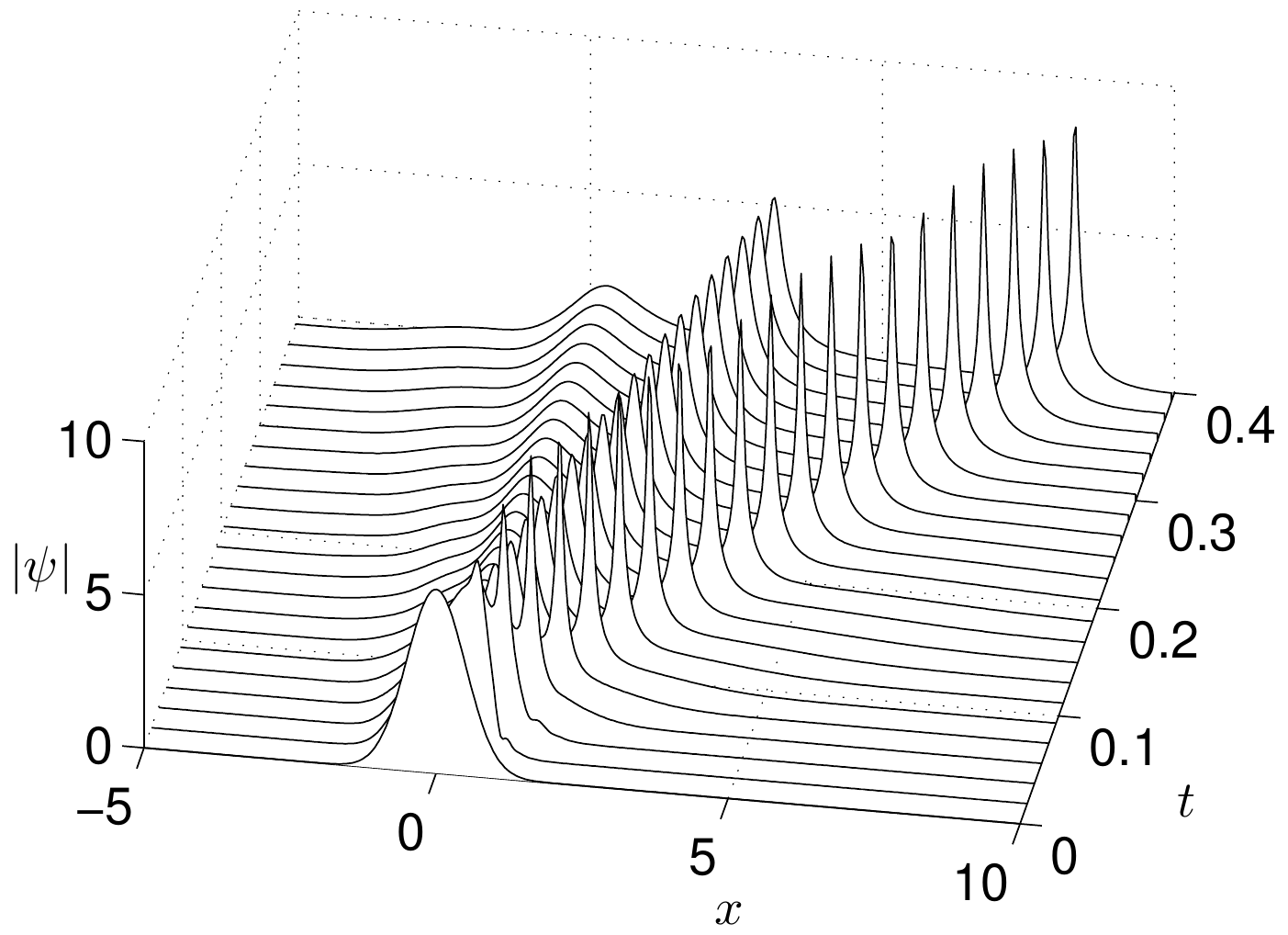}
    \caption{Time evolution of $|\psi(x,t)|$ for $\sigma =1 $.}
    \label{fig:A=3_water_fall_sigma=1}
  \end{center}
\end{figure}
\begin{figure}[h]
  \begin{center}
    \includegraphics[width=0.45\linewidth]{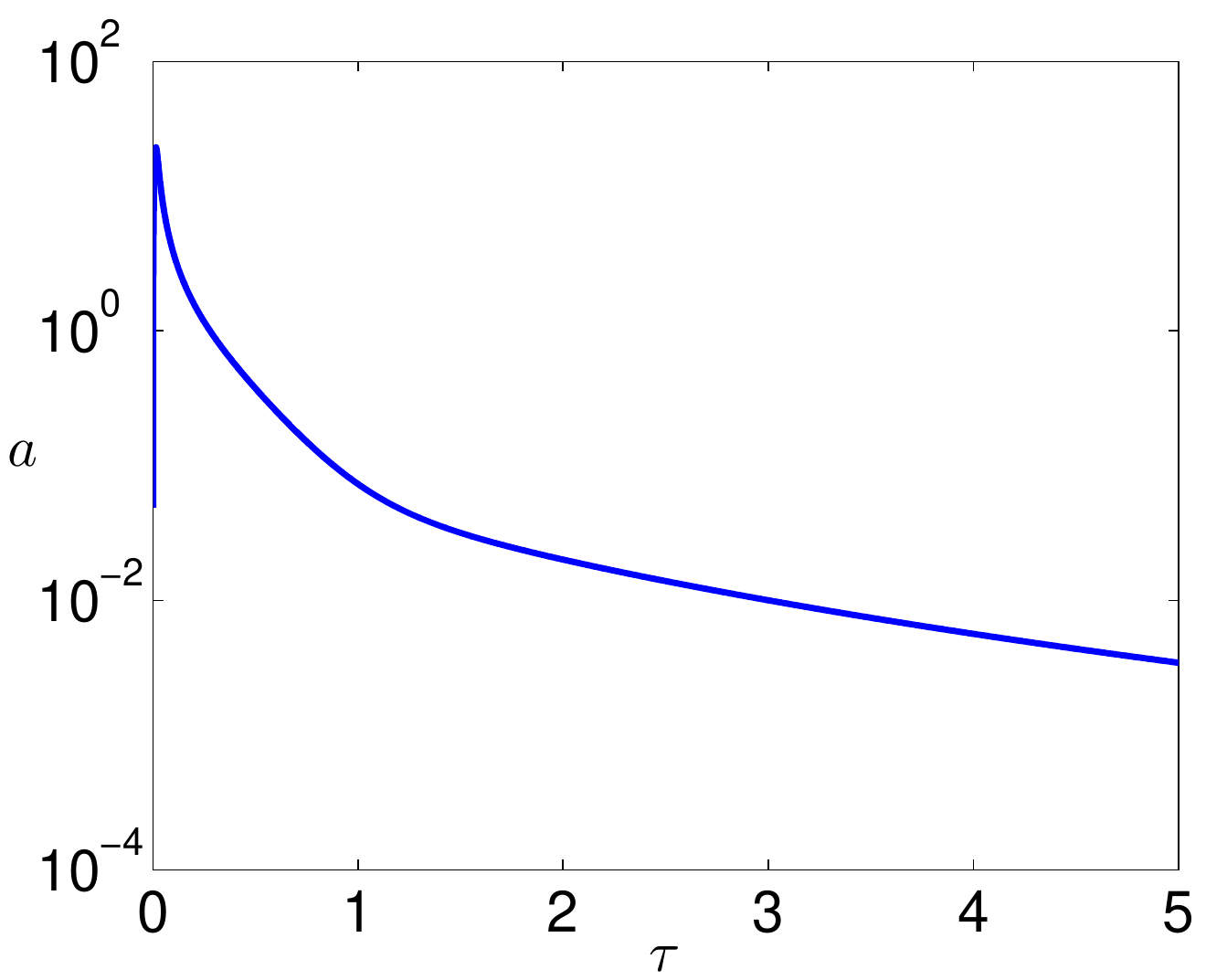}
    \includegraphics[width=0.45\linewidth]{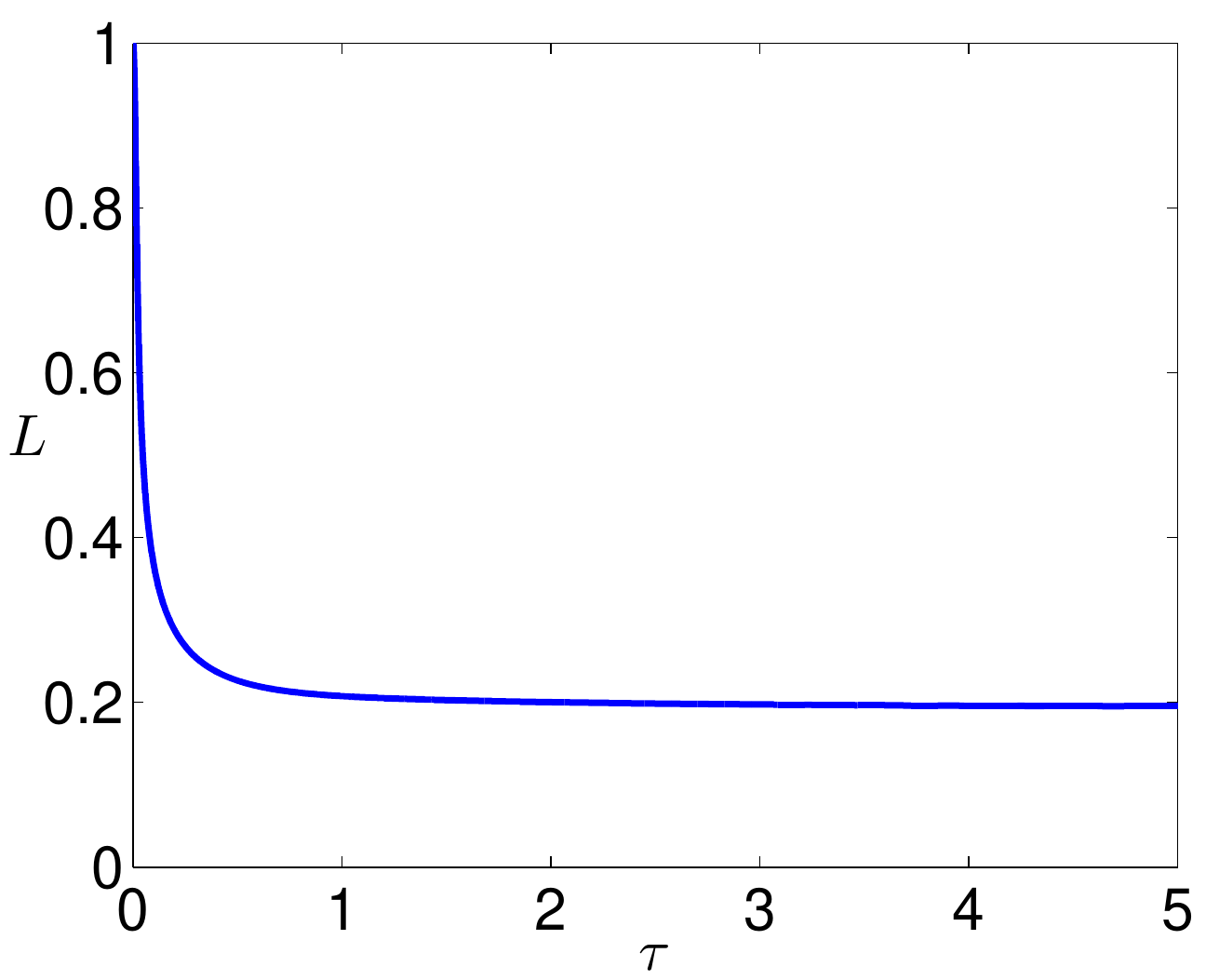}
    \caption{Time evolution of $a$ (Left) and $L$ (Right) for $\sigma
      =1$.}
    \label{fig:sigma=1_a}
  \end{center}
\end{figure}

\appendix
\section{Numerical Solution of the Boundary Value Problem}
\label{sec:BVP}

To integrate the nonlinear elliptic equation \eqref{e:blowup_soln} for
$Q$, we rewrite as a first order system for the four unknowns,
$\Re(Q), \Im(Q), \Re(Q_\xi)$ and $\Im(Q_\xi)$. As discussed before, we
solve the system in two adjacent domains $[-A_L,0]$ and $[0, A_R]$.
The solutions and their first order derivatives are matched by
continuity at $\xi=0$, while the boundary conditions at infinity are
of Robin type.  We denote
$$y_{1}^{\pm} = \Re(Q),  \, \, y_{2}^{\pm} = \Im(Q),\,\, y_{3}^{\pm} = \Re(Q_\xi), \,\; y_{4}^{\pm} = \Im(Q_\xi),$$
where $y^{-}$ and $y^{+}$ are the solutions in $[-A_L,0]$ and
$[0,A_R]$ respectively.  The system takes the form:
\begin{equation}
  \begin{split}
    \frac{d y_{1}^{\pm}}{d \xi}& = y_{3}^{\pm},\\
    \frac{d y_{2}^{\pm}}{d \xi}& = y_{4}^{\pm},\\
    \frac{d y_{3}^{\pm}}{d \xi}& = y_{1}^{\pm} +
    \alpha(\frac{1}{2\sigma} y_{2}^{\pm} +\xi y_{4}^{\pm})
    - \beta y_{4}^{\pm} +\left[(y_{1}^{\pm})^2 + (y_{2}^{\pm})^2\right]^\sigma y_{4}^{\pm},\\
    \frac{d y_{4}^{\pm}}{d \xi}& = y_{2}^{\pm} -
    \alpha(\frac{1}{2\sigma} y_{1}^{\pm} +\xi y_{3}^{\pm}) - \beta
    y_{3}^{\pm} -\left[(y_{1}^{\pm})^2 + (y_{2}^{\pm})^2\right]^\sigma
    y_{3}^{\pm}.
  \end{split}
\end{equation}
Solving four first order ODEs in two regions with two unknown
parameters requires imposing ten boundary conditions. Four of them are
the Robin boundary condition \eqref{BC-Q1} relating $\Re(Q), \Im(Q),
\Re(Q_\xi)$ and $\Im(Q_\xi)$ at $\xi = -A_L$ and $\xi = A_R$
respectively:
\begin{equation}
  \begin{split}
    -y_{1}^{-} - \frac{\alpha}{2\sigma} y_{2}^{-} - \alpha \xi y_{4}^{-} &= 0 \mbox{ at } \xi=-A_L,\\
    -y_{2}^{-} + \frac{\alpha}{2\sigma} y_{1}^{-} + \alpha \xi y_{3}^{-} &= 0\mbox{ at } \xi=-A_L,\\
    -y_{1}^{+} - \frac{\alpha}{2\sigma} y_{2}^{+} - \alpha \xi y_{4}^{+} &= 0 \mbox{ at } \xi=A_R,\\
    -y_{2}^{+} + \frac{\alpha}{2\sigma} y_{1}^{+} + \alpha \xi
    y_{3}^{+} &= 0\mbox{ at } \xi=A_R.
  \end{split}
\end{equation}
We impose the continuity of the solution at $\xi =0$:
\begin{equation}
  \begin{split}
    y_{1}^{+}(0) &= y_{1}^{-}(0),\\
    y_{2}^{+}(0) &= y_{2}^{-}(0),\\
    y_{3}^{+}(0) &= y_{3}^{-}(0),\\
    y_{4}^{+}(0) &= y_{4}^{-}(0).
  \end{split}
\end{equation}
The other two conditions are $|Q(0)|_\xi =0$ (the maximum value of $Q$
is attained at the origin) and $\Im(Q(0))=0$ (because the equation is
invariant by phase translation) :
\begin{equation}
  \begin{split}
    y_{1}^{-}y_{3}^{-} +y_{2}^{-}y_{4}^{-} &=0 \mbox{ at } \xi=0,\\
    y_{2}^{-} &=0 \mbox{ at } \xi=0
  \end{split}
\end{equation}

We proceed using a continuation method with respect to $\sigma$,
starting from $\sigma=2$ down to $\sigma=1.08$.  The MATLAB nonlinear
solver {\tt bvp4c} is used to integrate the system for each $\sigma$.
The two domains are automatically handled using the multipoint
feature, which permits for matching conditions at 0.  One needs to
provide a well-chosen initial guess.  We use the final
profile 
from our time-dependent simulation with the initial condition
\eqref{gauss}, $A_0 =3$ and $\sigma=2$, in the domain $[-10,10]$ to
extract the maximum bulk of the solution, (see Figures \ref{fig:Q} and
\ref{fig:Arg_Q}).  This solution is then extended to a larger domain
$[-150, 50]$ with an increment of 10. As shown in Figure
\ref{fig:continuation_ReQ_ImQ}, more and more oscillations occur as
$\sigma$ approaches $1$. The calculation becomes very delicate and we
use smaller and smaller increments of $\sigma$, namely $0.05, 0.01$
and $0.005$ for $\sigma$ in the intervals $2\geq\sigma\geq1.3,
1.3\geq\sigma\geq1.2$ and $1.2\geq\sigma\geq1.08$, respectively. We
set the maximum mesh points at $5\times10^5$ and the tolerance of the
difference between two iterations at $10^{-6}$.

\end{document}